\documentclass[a4paper, leqno]{article}
\usepackage{amsmath}
\usepackage{amsfonts}
\usepackage{amssymb}
\usepackage{graphicx}
\usepackage[all]{xy}
\usepackage{pdfsync}
\usepackage{mathrsfs}

\newtheorem{theorem}{Theorem}[section]
\newtheorem{theoremz}{Theorem}

\newtheorem{corol}[theorem]{Corollary}

\newtheorem{definition}[theorem]{Definition}
\newtheorem{example}[theorem]{Example}

\newtheorem{lemma}[theorem]{Lemma}

\newtheorem{prop}[theorem]{Proposition}
\newtheorem{remark}[theorem]{Remark}

\newenvironment{proof}[1][Proof]{\noindent\textbf{#1.} }{\ \rule{0.5em}{0.5em}}
\numberwithin{equation}{section}
\newcommand{\Rr}{\mathbb R}

\newcommand{\set}[1]{\left\{#1\right\}}

\newcommand{\eps}{\varepsilon}
\newcommand{\To}{\longrightarrow}
\newcommand{\rmap}{\longrightarrow}

\newcommand{\X}{\ensuremath{\mathfrak{X}}}

\renewcommand{\d}{\mathrm d}

\newcommand{\C}{\ensuremath{\mathrm{C}^{\infty}}}

\DeclareMathOperator{\GL}{GL}

\DeclareMathOperator{\pr}{pr}

\renewcommand{\hom}{\mathrm{Hom}}
\DeclareMathOperator{\modular}{mod}     
\renewcommand{\mod}{\modular}
\newcommand{\Jet}{\text{\rm J}}
\newcommand{\jet}{\text{\rm j}}
\newcommand{\ve}{\vartheta}          

\newcommand{\G}{\mathcal{G}}            
\renewcommand{\H}{\mathcal{H}}          
\newcommand{\A}{A}                      
\newcommand{\al}{\alpha}                
\newcommand{\be}{\beta}                 
\newcommand{\Lie}{\mathcal{L}}          
\newcommand{\Ad}{\text{\rm Ad}\,}       
\newcommand{\ad}{\text{\rm ad}\,}       
\newcommand{\tto}{\rightrightarrows}    

\newcommand{\B}{\mathfrak{g}}                    

\begin{document}
\title{Multiplicative forms and Spencer operators}
\author{Marius Crainic \and Maria Amelia Salazar \and Ivan Struchiner}
\date{}
\maketitle

\begin{abstract}
Motivated by our attempt to recast Cartan's work on Lie pseudogroups in  a more global and modern language, we are brought back to the question of understanding 
the linearization of multiplicative forms on groupoids and the corresponding integrability problem. From this point of view, the novelty of this paper is that we study forms {\it with
coefficients}. However, the main contribution of this paper is conceptual: 
the finding of the relationship between multiplicative forms and Cartan's work, which provides a completely new approach to integrability theorems for multiplicative forms. Back to Cartan, the multiplicative point of view shows that, modulo Lie's functor, the  
Cartan Pfaffian system (itself a multiplicative form with coefficients!) is the same thing as the classical Spencer operator.

\end{abstract}

\tableofcontents

\section{Introduction and main results}

This paper belongs to a longer project -- that of understanding Cartan's work on Lie pseudogroups \cite{Cartan1904, Cartan1905, Cartan1937} in a more global and modern language. 
An important role in our interpretation is played by Pfaffian systems on Lie groupoids. These are the main objects of this paper. 
We prove an integrability result, in the spirit of Lie, which allows us to pass from the (more interesting) global picture to the (easier-to-handle) linear picture. 
As an outcome we find that the associated infinitesimal data are certain ``connection-like operators", which we call `Spencer operators'. 
The main example is the classical Cartan system (represented by the canonical Cartan forms) on the groupoids of jets of diffeomorphisms
\cite{Cartan1904, Kuranishi, GuilleminSternberg:deformation, Olver:MC}. On the infinitesimal side we recover the classical Spencer operator \cite{Spencer, KumperaSpencer, Gold3,  GuilleminSternberg:deformation}. Hence, using 
Lie groupoids, we learn that the classical Cartan forms and the classical Spencer operators are the same thing, modulo the Lie functor.

As usual when dealing with Pfaffian systems on a manifold $\G$, there are two possible points of view (and two schools in the subject), dual to one another:
working with distributions $\H\subset T\G$ or, as Cartan, working with $1$-forms on $\G$; in the second case, while Cartan's considerations are local and involve a family of $1$-forms, the global formulation requires a $1$-form $\theta$ on $\G$ with coefficients in a vector bundle. The relationship between the two approaches is simply $\H= \textrm{Ker}(\theta)$. One advantage of Cartan's point of view is that it is slightly more general (it allows $\theta$ to have non-constant rank). A more important
advantage is that it allows for generalizations to differential forms of other degrees, leading to Cartan's exterior differential systems. On the other hand, the distribution point of view has the advantage that some integrability conditions become more natural and easier to handle globally. Both points of view will be present in this paper.

In order to talk about a ``Pfaffian system on the Lie groupoid $\G$'', there are some conditions that one has to impose on $\theta$ (or on $\H$); the most important one
is ``multiplicativity''- a compatibility condition with the groupoid multiplication. Actually, most of this paper is carried out under this condition alone. Also, although
our original exposition and proofs were using the language of distributions, we realised that the dual point of view allows us to treat (without extra effort) differential forms of
arbitrary degree. In conclusion, a large part of this paper is written for multiplicative forms with coefficients.

Working with forms of arbitrary degree is natural from Cartan's point of view (exterior differential systems). However, the main reason for us to allow general forms is the fact that
multiplicative $2$-forms are central to Poisson and related geometries. Moreover, while multiplicative $2$-forms with trivial
coefficients (!) are well understood, the question of passing from trivial to arbitrary coefficients has been around for a while. 
Surprisingly enough, even the statement of an integrability theorem for multiplicative forms with non-trivial coefficients was completely unclear.
This shows, we believe, that even the case of trivial coefficients was still hiding phenomena that was not understood. The fact that our work related to Lie pseudogroups clarifies this point
came as a nice surprise, and, looking back, we can now say what was missing from the existing picture in multiplicative $2$-forms: Spencer operators.\\

We would like to emphasize that the novelty of this paper lies not only in the main theorems, but also in the underlying approach. In particular, even in the known case of multiplicative $2$-forms with trivial coefficients, 
the proofs are completely new. Moreover, these ideas can be used in the study of other multiplicative structures (e.g. bivector fields, other tensor fields, distributional fields, etc).
The main results of this paper, explained in the Section \ref{sec: main results}, are the following:
\begin{itemize}
\item Theorem \ref{t1}: an integrability result which describes multiplicative forms with coefficients  (Definition \ref{def-mult-form}) in terms of their infinitesimal counterpart, i.e. Spencer operators (Definition \ref{def-Spencer-oprt}).
\item Theorem \ref{t2}: the dual of Theorem \ref{t1} in the case of $1$-forms, i.e. an integrability result for multiplicative distributions (Definition \ref{def-pf-syst}).
\item Theorem \ref{t3}: the infinitesimal characterization of the involutivity of multiplicative distributions.
\end{itemize}
Particular cases give rise to consequences which are interesting on their own (see Section \ref{sec: examples}). In particular, we obtain an integrability result for Cartan connections on groupoids and we describe a direct approach to the contact groupoids associated to Jacobi manifolds (with a slight generalization to the local Lie algebras of Kirillov \cite{Kirillov}). \\

Here are a few connections with the existing literature. On one hand, there is the literature related to Poisson geometry. Symplectic groupoids were discovered as the associated global objects \cite{Weinstein}, while Ping Xu realised the relevance of the multiplicativity condition \cite{Xu}. Multiplicative $0$-forms (1-cocycles) were studied in the context of quantization \cite{WeinsteinXu} (see also our Subsection \ref{1-cocyles and relations to the van Est map}). 
Motivated by Dirac geometry and the theory of Lie group-valued moment maps, the case of closed multiplicative $2$-forms was analyzed in \cite{BCWZ}. The case of multiplicative $1$-forms (not necessarily closed) appeared in \cite{prequantization} in the context of pre-quantization. General multiplicative forms, i.e., those which are not necessarily closed and are of arbitrary degree (but still with trivial coefficients!) were understood in \cite{BursztynCabrera, CrainicArias}. Distributions which are "multiplicative" in a sense more general than ours, but which are required to be involutive, were studied in \cite{Hawkins} in the context of geometric quantization and, more recently, in \cite{JotzOrtiz}. 
Moving towards Cartan's ideas, our Cartan connections from Subsection \ref{Cartan connections on groupoids} are the global counterpart of Blaom's Cartan algebroids \cite{Blaom}. 
The flat Cartan connections are the same ``flat connections on groupoids'' used by Behrend in the context of equivariant cohomology \cite{behrend}.
On the other hand, due to our approach, there is long list of literature on Lie pseudogroups and the geometry of PDE that serves as inspiration for our project (hence also for this paper) \cite{Cartan1904, Cartan1905, Cartan1937, Spencer, KumperaSpencer, GuilleminSternberg:deformation, Gold1, Gold2, Gold3, Olver:MC, Kamran, BC}. Of course, the appearance of the classical Cartan form and Spencer operator is an indication of this relationship with the theory of Lie pseudogroups; our operator (\ref{Theta}) when $k= 1$ corresponds to the start of the nonlinear Spencer complex {\it loc.cit} . Another indication is given in our Subsection \ref{Rough idea and some heuristics behind the proof}. The relationship with the geometry of PDE is more present in the case of our proof of Theorem \ref{t3}, but the detailed relationship will be explained elsewhere \cite{thesis}; however, we mention here that, implicitly, a central role is played by the notion of prolongation (and our cocycle $c$ from our Lemma \ref{cocycle} and its linearization $k$ are basically the same as the curvature $k$ which is central to \cite{Gold1} and appears in \cite{Gold2}, Prop. 8.3 in the non-linear case). Our finding that, modulo the Lie functor, the standard Cartan forms are the same thing as the classical Spencer operators may not be so surprising because both of them serve similar purposes (detect holonomic sections); however, from that point of view, our message is that all that the fundamental properties of these classical objects, on which everything else depends, are multiplicativity and the Spencer operator axioms. The fact that multiplicativity is fundamental for a more conceptual understanding of Cartan's structure equations is far less obvious; we realised that after staring for a few days at the formulas from the proof of Cartan's second fundamental theorem for Lie pseudogroups from \cite{Kamran} (pp. 59); but structure equations will be treated in a separate paper of our project.


\subsection*{Acknowledgments}
We would like to thank Henrique Burstyn for several interesting discussions throughout the development of this project.

\section{Basic definitions, first examples, main results}
\label{sec: main results}

\subsection{Multiplicative forms}
\label{Multiplicative forms}
Let $\G$ be a Lie groupoid over a manifold $M$; we will use the notation $\G \tto M$.
We will denote by $s,t: \G \to M$ the source and target maps of $\G$, by $u:M \to \G$ the unit map $u(x)= 1_x$, by
$i:\G \to \G$ the inversion $i(g) = g^{-1}$ and by $m: \G_2\to \G$ the multiplication $m(g, h)= gh$, defined on the
manifold of composable arrows
\[\G_2 = \set{(g,h)\in \G\times \G : s(g) = t(h)}.\]
We will also use the notation $g: x\rmap y$ to indicate that $s(g)= x$, $t(g)= y$. 
A representation 
of $\G$ is a vector bundle $\mu: E \to M$ and a smooth map 
\[\G \times_{s,\mu} E \longrightarrow E, \ (g, e)\mapsto g\cdot e\] 
defined on the fibered product of $\G$ and $E$ (i.e. each $g: x\rmap y$ defines a linear transformation $E_x\rmap E_y, v\mapsto g\cdot v$), which satisfies the usual 
identities of an action. Equivalently, a representation of $\G$ on a vector bundle $E$ is a groupoid homomorphism $\G\To \textrm{GL}(E)$, where $GL(E)$ is the Lie groupoid 
over $M$ whose arrows between two points $x, y\in M$ are the linear isomorphisms $E_x\stackrel{\sim}{\to} E_y$.

\begin{definition}\label{def-mult-form}\rm
Given a Lie groupoid $\G$ and a representation $E$ of $\G$, an \textbf{$E$-valued multiplicative $k$-form on $\G$} is any form $\theta \in \Omega^k(\G, t^{\ast}E)$ satisfying
\begin{equation}\label{eq: multiplicative}
(m^{\ast}\theta)_{(g,h)} = \pr_1^{\ast}\theta + g\cdot(\pr_2^{\ast}\theta)
\end{equation}
for all $(g,h) \in \G_2$, where $\pr_1,\pr_2: \G_2 \to \G$ denote the canonical projections. 
\end{definition}

\begin{example}[Linear forms; the classical linear Cartan form]\label{ex: linear forms} \rm A vector bundle $F \overset{\pi}{\to} M$ can be seen as a Lie groupoid with multiplication given by fiberwise addition (a bundle of abelian groups). In this case, any vector bundle $E$ over $M$ can be seen as a trivial representation of $F$ ($f\cdot e= e$). In this case, a multiplicative form $\theta\in \Omega^k(F, \pi^{\ast}E)$ is called a \textbf{linear form}. Thus, if $a: F\times_{\pi}F \to F$ denotes the addition of $F$, then $\theta$ is linear if and only if
\[a^{\ast}\theta = \pr_1^{\ast}\theta + \pr_2^{\ast} \theta.\]

An example is the linear Cartan 1-form associated to a vector bundle $E$, 
\[ \theta\in \Omega^1(\Jet^1E, E).\] 
Here $F= \Jet^1E\to M$ is the vector bundle consisting of first jets of sections of $E$. For the definition of $\theta$, consider $j^{1}_{x}s\in \Jet^1E$ (with $x\in M$, $s$ a section of $E$) and $\xi$ a tangent vector to $\Jet^1E$ at $j^{1}_{x}s$. Using the projection $\pr: \Jet^1E\to E$, 
\[ \theta(\xi):=\d\pr(\xi)- \d s(\d\pi(\xi)) \in T_{s(x)}E\]
is in the kernel of $\d\pi$, and hence it defines an element in $E_x$. 
\end{example}

\begin{example}[Classical Cartan form on the jet groupoids] \label{Classical Cartan form on the jet groupoids} \rm The classical Cartan 1-form on the groupoid $\Pi^1(M)$ consisting of first jets of local(ly defined) diffeomorphisms of $M$, with coefficients in $TM$, 
\[ \theta\in \Omega^1(\Pi^1(M), TM),\] 
is defined as follows. $\Pi^1(M)$ has source, target and multiplication given by
\[ s(j^{1}_{x}\phi)= x, \ s(j^{1}_{x}\phi)= \phi(x),\ j^{1}_{\phi(x)}\psi \cdot  j^{1}_{x}\phi= j^{1}_{x}\psi\circ\phi ,\]
and the action of $\Pi^1(M)$ on $TM$ is induced by the differential of diffeomorphisms. To describe $\theta$, consider $j^{1}_{x}(\phi)\in \Pi^1(M)$ (with $x\in M$, $\phi$ a diffeomorphism on $M$ defined around $x$) and $\xi$ a tangent vector to $\Pi^1(M)$ at $j^{1}_{x}\phi$. Then
\[ \theta(\xi):= \d t(\xi)- \d\phi ( \d s(\xi)) \in T_{\phi(x)}M .\]

Similarly, one has Cartan 1-forms on the higher jet-groupoids $\Pi^k(M)$, $\theta\in \Omega^k(\Pi^k(M), J^{k-1}TM)$. Both these Cartan forms, as well as the linear ones from the previous example, are particular instances of multiplicative Cartan forms on jet groupoids associated to a general groupoid (see the next section).
\end{example}

\begin{example}[Cohomologically trivial forms] \rm \label{ex: form on base}
Any form $\omega \in \Omega^k(M,E)$ induces a multiplicative form $\delta(\omega)\in \Omega^k(\G, t^{\ast}E)$:
\[\delta(\omega)_g = g\cdot s^{\ast}\omega - t^{\ast}\omega .\]
Forms of this type will be called cohomologically trivial (for indications of the terminology, see also Subsection \ref{1-cocyles and relations to the van Est map}). 


Note that the classical Cartan form $\theta\in \Omega^1(\Pi^1(M), TM)$ is of this type: it is $\delta(\iota)$, where $\iota\in \Omega^1(M, TM)$ is the identity of $TM$. For higher jets however, $\theta\in \Omega^k(\Pi^k(M), J^{k-1}TM)$ is not cohomologically trivial. 
\end{example}

\begin{remark}[Multiplicativity and bisections] \label{bisections}\rm \ Here is a point which, at least implicitly, is at the heart of our approach to multiplicative forms: they define ``pseudogroups of bisections of $\G$''.
Recall that a bisection of a Lie groupoid $\G$ over $M$ is any splitting $b: M\to \G$ of the source map with the property that $\phi_b:= t\circ b: M\rmap M$ is a diffeomorphism; 
the bisections of $\G$ form a group $\textrm{Bis}(\G)$ with the multiplication and the inverse given by
\[b_1\cdot b_2(x) = b_1(\phi_{b_2}(x))b_2(x),\ b^{-1}(x)=i\circ b\circ\phi_b^{-1}.\]
Local bisections are defined similarly, just that they are defined only over some open $U\subset M$ (and then $\phi_b$ is a diffeomorphism from $U$ to
$\phi_b(U)$); if $b_1$ is defined on $U_1$ and $b_2$ on $U_2$, then $b_1\cdot b_2$ is a local bisection defined on $\phi_{b_2}^{-1}( U_1)\cap U_2$. 

Given $\theta \in \Omega^k(\G, t^{\ast}E)$ multiplicative, one can talk about $\theta$-holonomic bisections of $\G$, i.e. those with the property that
$b^{\ast}\theta= 0$. The multiplicativity of $\theta$ ensures that the set $\textrm{Bis}_{\theta}(\G)$ of $\theta$-holonomic bisections is a subgroup of $\textrm{Bis}(\G)$.
\end{remark}


\subsection{Spencer operators}
\label{Spencer operators}

Passing to the infinitesimal picture, recall that a Lie algebroid over $M$ is a vector bundle $A \to M$ endowed with a vector bundle map (the anchor) $\rho: A \to TM$ and a Lie bracket on $\Gamma(A)$ satisfying the Leibniz identity
\[[\al,f\be] = f[\al,\be] + \Lie_{\rho(\al)}(f)\be\]
for all $\al,\be \in \Gamma(A)$ and $f \in \mathrm{C}^{\infty}(M)$. Here $\Lie_{\rho(\al)}$ is the Lie derivative along the vector field $\rho(\alpha)$. 
A  representation of a Lie algebroid $A$ is a vector bundle $E$ endowed with an $\Rr$-bilinear operator
which satisfies the usual connection-like identities:
\[ \nabla_{f\al}s = f\nabla_{\al}s, \ \ \nabla_{\al}(fs) = f\nabla_{\al}s +\Lie_{\rho(\al)}(f)s\]
(for all $\al \in \Gamma(A)$, $s\in \Gamma(E)$, and $f \in \mathrm{C}^{\infty}(M)$), and the the flatness condition
\[\nabla_{\al}\nabla_{\be} - \nabla_{\be}\nabla_{\al} = \nabla_{[\al,\be]}, \ \ \ \ \ \forall \  \al,\be \in \Gamma(A).\]
Each $\alpha\in \Gamma(A)$ induces a Lie derivative operator $\Lie_{\alpha}$ acting on $\Omega^k(M, E)$, which acts as $\Lie_{\rho(\alpha)}$ on forms and as $\nabla_{\alpha}$ on $\Gamma(E)$:
\[\Lie_{\al}\omega(X_1, \ldots, X_k) = \nabla_{\al}(\omega(X_1, \ldots, X_k)) - \sum_i\omega(X_1, \ldots, [\rho(\al),X_i], \ldots , X_k).\]

\begin{definition}\label{def-Spencer-oprt}\rm
Given a Lie algebroid $A$ over $M$ and a representation $E$ of $A$, an \textbf{$E$-valued $k$-Spencer operator on $A$} is a linear operator
\[ D: \Gamma(A) \to \Omega^k(M, E)\]
together with a vector bundle map 
\[ l:A \to \wedge^{k-1}T^{\ast}M\otimes E,\]
called the symbol of the Spencer operator, satisfying the Leibniz identity 
\begin{equation}\label{eq: Leibniz identity}
D(f\al) = fD(\al) + \d f\wedge l(\al)
\end{equation}
and the compatibility conditions:
\begin{eqnarray} 
 & D([\al,\be] ) = \Lie_{\al}D(\be) - \Lie_{\be}D(\al) \label{eq: compatibility-1}\\
 & l([\al,\be]) = \Lie_{\al}l(\be) - i_{\rho(\be)}D(\al) \label{eq: compatibility-2}\\
 & i_{\rho(\al)}l(\be) = -i_{\rho(\be)}l(\al), \label{eq: compatibility-3} 
\end{eqnarray}
for all $\al,\be \in \Gamma(A)$, and $f \in \C(M)$.
\end{definition}

\begin{remark} \rm \ When we are not in the ``special case'' $k= \textrm{dim}(M)+ 1$, the entire information is contained in $D$ and we only have to require 
(\ref{eq: compatibility-1}) and (\ref{eq: compatibility-3}). Indeed, in this case $l$ will be unique and (\ref{eq: compatibility-2}) follows from (\ref{eq: compatibility-1}) and Leibniz identities (plug in the first equation $f\beta$ instead of $\beta$ and expand using Leibniz). The fact that one has to adopt the previous definition so that it includes the ``special case'' $k= \textrm{dim}(M)+ 1$ is unfortunate because this case is not important for our
main motivating purpose (when $k= 1$). Keeping this in mind, we will simply say that $D$ is a Spencer operator and that $l$ is the symbol of $D$.
\end{remark}

\begin{example}[The classical Spencer operator]\label{1-jets-convenient}  \rm The classical Spencer operator associated to a vector bundle $E$
over $M$ is the unique linear operator
\[ D^{\textrm{clas}}: \Gamma(\Jet^1E) \To \Omega^1(M, E),\]
satisfying the Leibniz identity relative to $\pr: \Jet^1E\to E$, as well as the 
holonomicity condition
\[ D^{\textrm{clas}}(j^1(s))= 0 \ \ \ \forall \ s\in \Gamma(E).\]
Of course, viewing $\Jet^1E$ as a trivial Lie algebroid (zero anchor, zero bracket), with the trivial action on $E$, $D^{\textrm{clas}}$ is an $E$-valued Spencer operator. 

One can define $D^{\textrm{clas}}$ using the linear Cartan form from Example \ref{ex: linear forms}, as
\[ D^{\textrm{clas}}(s)= s^{\ast}\theta .\]
Alternatively, $D$ is the second component of a canonical identification 
\begin{equation}\label{J-decomposition} 
\Gamma(\Jet^1E)\cong \Gamma(E)\oplus \Omega^1(M, E),
\end{equation}
which we will call the Spencer decomposition, and which provides a convenient way of representing the sections of first jet bundles.
This decomposition comes from the short exact sequence of vector bundles 
\[ 0\to \textrm{Hom}(TM, E)\stackrel{i}{\to} J^1(E) \stackrel{\pr}{\to} E\to 0\]
where $\pr$ is the canonical projection $j^{1}_{x}s\mapsto s(x)$ and $i$ is determined by
\[ i(\d  f\otimes \alpha)= f\jet^1(\alpha)- \jet^1(f\alpha).\]
Although this sequence does not have a canonical splitting, at the level of sections it does: $\alpha\mapsto \jet^1(\alpha)$.
This gives the identification (\ref{J-decomposition}). In other words, any $\xi\in \Gamma(\Jet^1E)$ can be written uniquely as
\[  \xi= j^1(\alpha)+ i(\omega)\]
with $\al\in \Gamma(A)$, $\omega\in \Omega^1(M, A)$; we write $\xi= (\alpha, \omega)$. One should keep in mind however that the resulting $C^{\infty}(M)$-module structure becomes
\[ f \cdot (\alpha, \omega)= (f\alpha, f\omega+ \d f\wedge \alpha).\]
In terms of the Spencer operator, $\xi= (\pr(\xi), D^{\textrm{clas}}(\xi))$ 
and the module structure gives the Leibniz identity for $D^{\textrm{clas}}$.

Note that, again, there is a version of $D^{\textrm{clas}}(\xi)$ on higher jets:
\[ D^{\textrm{clas}}: \Gamma(J^kE) \To \Omega^1(M, J^{k-1}E).\]
\end{example}

\begin{remark}\label{rk-convenient}  \rm \ 
Note that in the Pfaffian case ($k= 1$), a general $E$-valued Spencer operator as in the previous definition can be encoded/interpreted in a vector bundle map
\[ j_{D}: A\To \Jet^1E .\]
Indeed, the Leibniz condition for $D$ relative to $l$ is equivalent to the fact  
$(l, D): \Gamma(A)\To \Gamma(E)\oplus \Omega^1(M, E)$
is $C^{\infty}(M)$-linear with respect to the module structure on the right hand side mentioned above. Hence, using the identification (\ref{J-decomposition}), we see that we deal with a morphism of vector bundles $j_{D}$ as above.

Note that $D$ itself can be recovered as the composition $D^{\textrm{clas}}\circ j_{D}$. For the classical Spencer operator, it corresponds to $j_{D^{\textrm{clas}}}= \textrm{Id}$. 
\end{remark}

\begin{example}\label{ex: form on base-2}\rm \ Here is the infinitesimal analogue of the cohomologically trivial forms of Example \ref{ex: form on base}: for any algebroid $A$ and any representation $E$ of $A$, any form $\omega\in \Omega^k(M, E)$ induces an $E$-valued Spencer operator by 
 \[D(\al) = \Lie_{\al}\omega, \quad l(\al) = i_{\rho(\al)}\omega .\]
\end{example}

\subsection{The Lie functor: integrability (Theorem \ref{t1})}
\label{The Lie functor}

In this paper the term ``Lie functor'' is used to indicate the passing from the global picture (groupoids) to the infinitesimal picture (algebroids) and should be thought of as ``linearization''. The reverse process is coined as ``integration''. 

The first ``example'' is the construction of the Lie algebroid $A= A(\G)$ of a Lie groupoid $\G$ (over a base manifold $M$). Recall that, 
as a vector bundle over $M$, $A= u^{\ast}T^{s}\G$ is the pullback by the unit map of the vector bundle of vectors tangent to the $s$-fibers.
Using right translations, the space of sections $\Gamma(A)$ is identified with the space of right-invariant vector fields on $\G$, and so the Lie bracket of vector fields induces a Lie bracket $[\cdot, \cdot]$ on sections of $\Gamma(A)$
(see also the remark below). Finally, $\rho = \d t \vert_{A}$. 

For the reverse process, starting with a Lie algebroid $A$, one looks for a Lie groupoid $\G$ which integrates $A$, i.e. whose Lie algebroid is isomorphic to $A$; if such a $\G$ exists, one says that $A$ is integrable. 
A basic result  in the theory of Lie groupoids states that, for an integrable Lie algebroid $A$, one finds an unique (up to isomorphisms) Lie groupoid $\G$ which integrates $A$ and which is $s$-simply connected (in the sense that all the 
fibers of $s: \G\to M$ are connected and simply connected). 

Given a Lie groupoid $\G$ with Lie algebroid $A$, intuitively, the Lie functor takes structures on $\G$ and transforms them into structures on $A$. 
It is good to keep in mind that, for the reverse process (integrability), the $s$-simply connectedness of $\G$, mentioned above, constantly appears as a necessary condition.

For instance, any representation $E$ of $\G$ becomes a representation of $A$ as follows: for $\al \in \Gamma(A)$, and $e \in \Gamma(E)$,
\begin{eqnarray}
\label{rep}\nabla_{\al}e(x) = \frac{\d}{\d \eps}\bigg\vert_{\eps = 0}g(\eps)^{-1}\cdot e\big(t(g(\eps))\big),
\end{eqnarray}
where $g(\eps)$ is any curve in $s^{-1}(x)$ with $g(0)=1_x$, $\frac{\d}{\d \eps}\vert_{\eps = 0}g(\eps) = \al(x)$. Under the $s$-simply connectedness assumption on $\G$, 
one finds that this construction defines a 1-1 correspondence between representations of $\G$ and representations of $A$. Our first main result is a similar correspondence between multiplicative forms and 
Spencer operators.

\begin{remark}[Right invariance and flows]\label{right-invariance} \rm \ For explicit formulas, it is useful to be more explicit about the identification of $\Gamma(A)$ with right invariant vector fields and about their induced flows. The right translations by an
element $g: x\to y$ of $\G$ are 
\[ R_g: s^{-1}(y)\To s^{-1}(x), \ R_g(a)= ag .\]
At the level of tangent vectors, one has to restrict to the bundle $T^s\G= \textrm{Ker}(ds)$ of vectors tangent to the $s$-fibers; we denote by the same letter $R_g$ the induced linear maps, obtained by differentiation (going from 
$T^{s}_{a}\G$ to $T^{s}_{ag}\G$ for $a\in s^{-1}(y)$). With this, the space of right invariant vector fields on $\G$ is
\[ \X^{\textrm{inv}}(\G)= \{X\in \Gamma(T^s\G): R_g(X_a)= X_{ag} \ \forall\ a, g\in \G\ \textrm{composable} \}.\]
The identification $\Gamma(A)\stackrel{\sim}{\To} \X^{\textrm{inv}}(\G)$ sends $\alpha\in \Gamma(A)$ to $\alpha^r\in \X^{\textrm{inv}}(\G)$ given by
\[ \alpha^{r}_{g}= R_g(\alpha_{t(g)}).\]

For $\alpha\in \Gamma(A)$, on defines the (local) flow of $\alpha$ by
\[ \phi_{\alpha}^{\epsilon}:= \varphi_{\alpha^r}^{\epsilon}|_{M}: M\To \G ,\]
where $\varphi_{\alpha^r}^{\epsilon}$ is the (local) flow of the right invariant vector field $\alpha^r$. As usual, we are sloppy with the precise notations for the domain of the flow.
From right invariance it follows that $\phi_{\alpha}^{\epsilon}$ is a bisection of $\G$ (see Remark \ref{bisections}) which determines
the entire flow $\varphi_{\alpha^r}^{\epsilon}$ ($\varphi_{\alpha^r}^{\epsilon}(g)= \phi_{\alpha}^{\epsilon}(t(g))g$). Note also that, in terms of multiplication of (local) bisections
(Remark \ref{bisections} again), the flow property for  $\varphi_{\alpha^r}^{\epsilon}$ translates into
\[ \phi_{\alpha}^{\epsilon}\cdot \phi_{\alpha}^{\epsilon'}= \phi_{\alpha}^{\epsilon+ \epsilon'}.\]
This shows that, morally, $\Gamma(A)$ plays the role of the Lie algebra of $\textrm{Bis}(\G)$. 
\end{remark}

\begin{theoremz}\label{t1}
Let $E$ be a representation of a Lie groupoid $\G$ and let $A$ be the Lie algebroid of $\G$. Then any multiplicative form $\theta \in \Omega^k(\G, t^{\ast}E)$ induces an
$E$-valued Spencer operator $D_{\theta}$ of order $k$ on $A$, given by 
\begin{equation}\label{eq: explicit formula}\small{
\left\{\begin{aligned}
D_{\theta}(\al)_x(X_1, \ldots, X_k) &= \frac{\d}{\d \eps}\Big|_{\eps = 0} \phi^{\eps}_{\al}(x)^{-1}\cdot\theta((\d \phi^{\eps}_{\al})_x(X_1), \ldots, (\d \phi^{\eps}_{\al})_x(X_k)), \\ \\l_{\theta}(\al) &= u^{\ast}(i_{\al}\theta) .\end{aligned}\right.}
\end{equation}

If $\G$ is $s$-simply connected, then this construction defines a 1-1 correspondence between  $E$-valued $k$-forms on $\G$ and 
$E$-valued $k$-Spencer operators on $A$.  
\end{theoremz}

 \begin{remark}\rm 
Let us look again at the case when $A= F$ is a Lie algebroid with trivial bracket and anchor; then Theorem \ref{t1} gives a one-to-one correspondence between linear forms $\theta \in \Omega^k(F,\pi^{\ast}E)$ and operators $D: \Gamma(F) \to \Omega^k(M, E)$, with a symbol map $l:F \to \wedge^{k-1}T^{\ast}M\otimes E$, satisfying Leibniz equation (all compatibility conditions are automatically satisfied). 

This actually indicates a possible strategy, in the spirit of \cite{BursztynCabrera}, but which we will not follow here, for proving Theorem \ref{t1}: given $\theta\in\Omega^k(\G, t^{\ast}E)$, first ``linearize'' $\theta$ to a linear form $\theta_{0}\in \Omega^k(A, t^{\ast}E)$ then consider the associated Spencer operator $D$, carefully book-keeping all the equations involved. 
\end{remark}

\begin{example}[Classical Cartan forms/Spencer operators] \rm It is not difficult to see (and will be explained in full generality in the next section) that the
Spencer operator associated to the linear Cartan form $\theta\in \Omega^1(\Jet^1E, E)$ (Example \ref{ex: linear forms}) is precisely the classical Spencer operator
of the vector bundle $E$; similarly, the one associated to the Cartan form $\theta\in \Omega^1(\Pi^1(M), TM)$ (Example \ref{Classical Cartan form on the jet groupoids}) is the same classical Spencer operator but interpreted in the algebroid context.

Hence our main theorem (Theorem \ref{t1}) gives the precise relationship between the classical Cartan forms and Spencer operators: modulo the Lie functor, they are one and the same thing.
\end{example}

\begin{example}\rm
Here is another ``baby example''. Recall that an Ehresmann connection on a vector bundle $F$ is a splitting $\sigma: \pi^{\ast}TM \to TF$ of the exact sequence of vector bundles over $F$,
\[\xymatrix{0\ar[r]&T^\pi F \ar[r] &TF \ar[r]^-{\d\pi} &\pi^{\ast}TM\ar[r]&0 }\]
where $T^\pi F = \ker \d\pi$ denotes the bundle of vectors tangent to the fibers of $F$. 

Since $T^\pi F$ is canonically isomorphic to $\pi^{\ast}F$, it follows that an Ehresmann connection $\sigma$ is the same as a 1-form $\theta_{\sigma} \in \Omega^1(F, \pi^{\ast}F)$. The form $\theta_{\sigma}$ is linear if and only if $\sigma$ is a linear connection. Thus, in this case, our theorem reduces to the well-know correspondence between linear Ehresmann connections on a vector bundle $F \to M$ and covariant derivative operators 
\[D = \nabla: \X(M) \times \Gamma(F) \to \Gamma(F).\]
In this case the bundle map $l: F \to F$ is just the identity.
 \end{example}

\subsection{The Pfaffian case; dual version (Theorem \ref{t2})}
\label{The Pfaffian case}
We concentrate now on the Pfaffian case ($k= 1$). The usual duality between $1$-forms and distributions admits a multiplicative version (see Subsection \ref{sec: distributions}), giving rise to a dual version of Theorem \ref{t1} (case $k= 1$). Here is the outcome.


To discuss multiplicativity of distributions recall that one has a Lie groupoid $T\G\rightrightarrows TM$ associated to any Lie groupoid $\G\rightrightarrows M$; its structure maps are just the differentials of the structure
maps of $\G$. 

\begin{definition}\label{def-pf-syst} 
A multiplicative distribution on $\G$ is any distribution $\H\subset T\G$ which is also a Lie subgroupoid of $T\G\rightrightarrows TM$ (with the same base $TM$).
\end{definition}

Note that the fact that $\H$ is a Lie groupoid over $TM$ implies that $\H$ is transversal to the $s$-fibers of $\G$ and that the $s$-vertical part of $\H$,
\[ \H^s:= \H\cap T^s\G \] 
has constant rank. Restricting to $M$, one obtains a sub-bundle of $A= \textrm{Lie}(\G)$, 
\[ \B:= \H^{s}|_{M}\subset (T^s\G)|_{M}= A,\]
which is an important piece of the infinitesimal data associated to $\H$. Borrowing the terminology from the theory of EDS \cite{Gold1, Gold2, BC}, we will call it {\it the symbol space} of $\H$. We will also consider the quotient
\[ E= A/\B.\]
On the infinitesimal side, remark that condition (\ref{eq: compatibility-2}) in the definition of Spencer operators for $k= 1$ implies that the operator $\nabla$ which makes $E$ into a representation of $A$ can be recovered from $D$, hence one has a slight reformulation of Definition \ref{def-Spencer-oprt} in this case. Here we also pass to the notation
\[ D(\alpha)(X)= D_{X}(\alpha).\]

\begin{definition}\label{relative-Spencer}
Let $A$ be a Lie algebroid over $M$, let $\B$ a sub-bundle of $A$ and consider $E:= A/\B$ with the quotient map denoted 
$l: A\to E$. A Spencer operator on $A$ relative to $\B$ (or relative to $l$), is any $\Rr$-bilinear map 
\[ D:\X(M)\times\Gamma(A)\to\Gamma(E),\ \ (X, \alpha)\mapsto D_X(\alpha)\]
 which is $C^\infty(M)$-linear in $X$, satisfies the Leibniz identity relative to $l$: 
\begin{eqnarray*}
D_X(f\alpha)=fD_X\alpha+\Lie_X(f)l(\alpha) 
\end{eqnarray*}
and the following two compatibility conditions
\begin{eqnarray}\label{vertical}
D_{\rho(\beta)}\alpha=-l[\alpha, \beta],
\end{eqnarray} 
\begin{eqnarray}
\label{vertical2}
\qquad D_X[\alpha,\alpha']=\nabla_\alpha(D_X\alpha')-D_{[\rho(\alpha),X]}\alpha'-\nabla_{\alpha'}(D_X\alpha)+D_{[\rho(\alpha'),X]}\alpha,
\end{eqnarray}
for all $\alpha, \al'\in \Gamma(A)$, $\be\in \Gamma(\B)$, $X\in\mathfrak{X}(M)$, $f\in C^{\infty}(M)$ and where
\begin{equation}\label{nabla-out-of-D}
\nabla:\Gamma(A)\times \Gamma(E)\to \Gamma(E), \ \nabla_\alpha(l(\alpha'))=D_{\rho(\alpha')}\alpha+l[\alpha,\alpha'].
\end{equation}
\end{definition}

\begin{remark} \label{remark-dual}\rm Condition (\ref{vertical})  implies that $\nabla$ is well-defined; also, (\ref{vertical2}) applied to $X\in \textrm{Im}(\rho)$ implies 
the flatness condition for $\nabla$, so that $E$ becomes a representation of $A$. We see that the previous definition is just a reformulation of the notion of $1$-Spencer operator in the case when the symbol map is surjective .
Indeed (\ref{eq: compatibility-2}) becomes our (\ref{nabla-out-of-D}) and also implies (\ref{vertical}) from the previous definition;
also, (\ref{vertical2}) is just (\ref{eq: compatibility-1}) for $k= 1$. 
\end{remark}

\begin{theoremz}\label{t2} Let $\G\rightrightarrows M$ be an $s$-simply connected Lie groupoid with Lie algebroid $A\to M$. There is a one to one correspondence between 
\begin{enumerate}
\item multiplicative distributions $\H\subset T\G$, 
\item sub-bundles $\B\subset A$ together with a Spencer operator on $A$ relative to $\B$.
\end{enumerate}

In this correspondence, $\B$ is the symbol space of $\H$ and 
\begin{eqnarray}\label{D}
D_X\alpha(x)=[\tilde X,\alpha^r]_x\mod \H^s_{1_x} ,
\end{eqnarray}
where $\tilde X\in \Gamma(\H)\subset\mathfrak{X}(\G)$ is any vector field which is $s$-projectable to $X$ and extends $u_*(X)$ (for $\alpha^r$, see Remark \ref{right-invariance}). 
\end{theoremz}

\subsection{Involutivity (Theorem \ref{t3})}
\label{Involutivity}

Let now $\H$ be a multiplicative distribution on a Lie groupoid $\G$ and consider the
associated symbol space $\B= \H^{s}|_{M}$, the representation $E= A/\B$, and the associated Spencer operator
\[ D: \X(M)\times \Gamma(A)\to \Gamma(E).\] 
From the Leibniz identity for $D$ we obtain that $D_{X}(\beta)$ is $C^{\infty}(M)$-linear
on both arguments for $\beta\in \Gamma(\B)$. Hence we obtain a vector bundle map
\[ j_{\B}: \B\to \textrm{Hom}(TM, E),\]
called the symbol representation. Note that, in terms of the jet-representation of Spencer operators 
(Example \ref{1-jets-convenient}), this is just the restriction of $j_{D}$ to $\B$. 
Remark that, if $j_{\B}= 0$, then $D$ induces a connection
\[\nabla^{E}: \X(M)\times \Gamma(E)\to \Gamma(E) ,\ \nabla^{E}_{X}[\al]= D_X(\al).\]

\begin{theoremz}\label{t3}
A multiplicative distribution $\H \subset T\G$ is involutive if and only if the symbol representation $j_{\B}$ vanishes
and the connection $\nabla^{E}$ on $E$ is flat.
\end{theoremz}

\begin{example}\label{ex-flat-Cartan} \rm \ Let $\rho: \mathfrak{h}\to \X(M)$ be an infinitesimal action of a Lie algebra $\mathfrak{h}$ on $M$. One has an associated Lie
algebroid $\mathfrak{h}\ltimes M$, which as a vector bundle is the trivial one with fiber $\mathfrak{h}$ (so $\Gamma(\mathfrak{h}\ltimes M)= C^{\infty}(M, \mathfrak{h})$), the anchor is the infinitesimal action and the
bracket is uniquely determined by the Leibniz identity and the condition that, on constant sections $u, v\in \mathfrak{h}$, it coincides with the bracket of $\mathfrak{h}$. In this case
the canonical flat connection
\[ \nabla^{\textrm{flat}}: \X(M)\times C^{\infty}(M, \mathfrak{h})\To C^{\infty}(M, \mathfrak{h})\]
satisfies the conditions from the previous theorem with $E= \mathfrak{h}\ltimes M$, $l= \textrm{Id}$. Hence one obtains a flat involutive $\H$ on the integrating groupoid.
This can be best seen when the infinitesimal action comes from the action of a Lie group $H$ on $M$. Then $\mathfrak{h}\ltimes M$ is the Lie algebroid of the action groupoid 
$H\ltimes M$, which is the manifold $H\times M$ with the groupoid structure 
\[ s(g, x)= x, \ t(g, x)= gx,\ (g, hy)\cdot (h, y)= (g, y).\]
The flat involutive $\H$ on $H\times M$ is simply the foliation with the leaves $\{h\}\times M$ (for $h\in H$). See also Corollary \ref{flat-Cartan}.
\end{example}

\begin{corol} If $\G$ is $s$-simply connected then there is a 1-1 correspondence between
\begin{enumerate}
\item[1.] involutive multiplicative distributions $\H$ on $\G$.
\item[2.] a flat vector bundle $(E, \nabla^{E})$ over $M$, a $\nabla^{E}$-parallel tensor $T: \Lambda^2E\to E$ and a surjective 
vector bundle map $l: A\to E$ satisfying 
\[ l([\alpha, \beta])= \nabla^{E}_{\rho(\alpha)}(l(\beta))- \nabla^{E}_{\rho(\beta)}(l(\alpha))+ T(l(\alpha), l(\beta)), \ \ \forall\ \alpha, \beta\in \Gamma(A).\]
\end{enumerate}
\end{corol} 

\begin{proof} The last equation defines $T$ in terms of $\nabla^{E}$ and $l$; the only problem is whether it is well-defined, but this immediately follows from (\ref{vertical}). 
The rest follows from the fact that $D$ is determined by $\nabla^{E}$ (a condition that itself implies that $j_{\mathfrak{g}}= 0)$. Hence one just has to 
rewrite the equation (\ref{vertical2}) in terms of $\nabla^{E}$ and $T$, and one finds the condition that $T$ is $\nabla^E$-parallel. 
\end{proof}


\section{Examples}\label{sec: examples}
\subsection{Jet groupoids}\label{sec: Jet groupoids}
Our motivating example comes from Cartan forms on jet groupoids (and subgroupoids of them). Let $\G$ be a Lie groupoid over $M$. 
Its first jet groupoid, $\Jet^1\G$, consists of 1-jets $\jet^1_xb$ of local bisections of $\G$, with the groupoid structure given by
(where, for bisections and their multiplication, see Remark \ref{bisections}):
\begin{eqnarray*}
s(\jet^1_xb)= x,\ t(\jet^1_xb)= \phi_b(x)
\end{eqnarray*}
\begin{eqnarray*}
\jet^1_{\phi_{b_2}(x)}b_1\cdot\jet^1_xb_2=\jet^1_x(b_1\cdot b_2),\text{ and }(\jet^1_xb)^{-1}=\jet_{\phi_{b}(x)}^1(b^{-1}).
\end{eqnarray*}

Of course, any (local) bisection $b$ of $\G$ induces a (local) bisection $j^1b$ of $\Jet^1\G$ given by $x\mapsto j^{1}_{x}b$; bisections
of $\Jet^1\G$ of this type are called holonomic. The canonical Cartan form on $\Jet^1\G$,
\[ \theta_{\mathrm{can}}\in \Omega^1(\Jet^1\G, t^{\ast}A),\]
is designed to detect the bisections $\zeta$ of $\Jet^1\G$ which are holonomic: the condition is $\zeta^{\ast}\theta_{\mathrm{can}} = 0$.
We recall here the explicit description of $\theta_{\mathrm{can}}$. Let $\pr:\Jet^1\G\to \G$ be the canonical projection and let $\xi$ be a vector tangent to $\Jet^1\G$ at some point
$\jet^1_xb\in \Jet^1\G$. Then the difference
\[(\d \pr)_{\jet^1_xb}(\xi)-(\d b)_x (\d s)_{\jet^1_xb}(\xi) \in T_g\G\] 
is killed by $\d s$, hence it comes from an element in $A_{t(g)}$:
\[\theta_{\mathrm{can}}(\xi)=R_{b(x)^{-1}}((\d \pr)_{\jet^1_xb}(\xi)- (\d b)_x (\d s)_{\jet^1_xb}(\xi))\in A_{t(g)}.\]

\begin{example}\rm The linear Cartan form of Example \ref{ex: linear forms} and the classical Cartan 1-form of Example \ref{Classical Cartan form on the jet groupoids} 
are of this type. In the second case, $\G$ is the pair groupoid of $M$, i.e. $M\times M$, with $s= \pr_1$, $t= \pr_2$, and multiplication $(y, z)\cdot (x, y)= (x, z)$.  
Note also that, in this case, $\textrm{Bis}(\G)= \textrm{Diff}(M)$. 
\end{example}

Observe that the correspondence which associates to an element $\sigma= \jet^{1}_{x}(b) \in \Jet^1\G$ the isomorphism $\lambda_{\sigma}:= (d\phi_b)_x: T_{s(\sigma)}M \to T_{t(\sigma)}M$ determines a representation of $\Jet^1\G$ on $TM$. Similarly, one has a representation of $\Jet^1\G$ on $A$, called the \textbf{adjoint representation}, as follows. A bisection $b$ of $\G$ acts on $\G$ by conjugation
\[\Ad_b(g) = b(t(g))\cdot g \cdot b(s(g))^{-1}.\]
It is clear that this action maps units to units, and source fibers to source fibers. Moreover, the differential of $\Ad_b$ at a unit $x$ depends only on $\jet^1_xb$. We define
\[\Ad_{\jet^1_xb}\al = (\d\Ad_b)_x(\al) \ \ \ (\al \in A_x).\]

\begin{remark}[when working with $\Jet^1\G$]\rm \ \label{when working with jets}
Here is a slightly different description of $\Jet^1\G$, which we will be using whenever we have to work more explicitly with $\Jet^1\G$. 
Since a first jet $\jet^1_xb$ of a bisection $b$ at $x\in M$ is encoded in $g:= b(x)$ and $d_xb: T_xM\to T_g\G$, we see that an element of $\Jet^1\G\tto M$ can be thought of as 
a pair $(g, \sigma)$ where $g\in \G$ and 
\begin{eqnarray*}
\sigma:T_xM\to T_g\G
\end{eqnarray*}
is a splitting of the map $(\d s)_x:T_g\G\to T_xM$ with the property that
\begin{eqnarray}\label{comp}
\lambda_{\sigma}:= \d t\circ\sigma:T_xM\to T_{t(g)}M \text{ is an isomorphism.}
\end{eqnarray}
Of course, $g= \pr(\sigma)$, but we will often use the notation $\sigma_g$ to indicate $g$. 
The groupoid structure of $J^1(\G)$ becomes:
\begin{equation}\label{mult-i-J1}
 s(\sigma_g)= s(g),\ \ t(\sigma_g)= t(g),  \ \ \sigma_g\cdot \sigma_h (u)=(\d m)_{(g, h)}(\sigma_g(\lambda_{\sigma_h}(u)), \sigma_h(u)) .
\end{equation}

With these the adjoint representations of $\Jet^1\G$ on $A$ becomes 
\begin{equation}\label{eq: Ad}
\Ad_{\sigma_g}\al = R_{g^-1} (\d m)_{(g, s(g))}(\sigma_g(\rho(\al)), \al)\in A_y
\end{equation}
for $g: x\to y$, $\alpha\in A_x$ (this follows using the flow of $\alpha$ to compute $(\d\Ad_b)_x(\al)$).
\end{remark}

Let us move to the infinitesimal side of the discussion.  Recall that, for any Lie algebroid $A$ over $M$,
first jets of sections of $A$ form a new algebroid $\Jet^1A$ over $M$, where the anchor is the composition of the anchor of $A$ with the canonical projection
$\pr: \Jet^1A \to A$, and where the bracket is uniquely determined by the Leibniz identity and the condition that
\[[\jet^1\al,\jet^1\be] = \jet^1[\al,\be] \]
for any two sections $\al, \be$ of $A$ (see below for a formula on general sections).

Moreover, $\Jet^1A$ has a canonical adjoint representations on $TM$ and on $A$ (both denoted by $\ad$) determined by the Leibniz  identities and the conditions
\[\ad_{\jet^1\al}X = [\rho(\al),X], \text{ and } \ad_{\jet^1\al}\be = [\al,\be].\]
If $A$ is the Lie algebroid of a Lie groupoid $\G$, then $\Jet^1A$ is the Lie algebroid of $\Jet^1\G$, and these representations correspond to the canonical representations of $\Jet^1\G$ on $TM$ and $A$ respectively.

\begin{remark}\rm \ The formulas for the bracket on $\Jet^1A$ and for the actions can be written for general sections of $\Jet^1A$ using the classical Spencer operator associated to the vector bundle $A$ (see Example \ref{1-jets-convenient}). For the actions,
\[ \nabla_{\xi}(X)= [\rho(\xi), X]+ \rho D_{X}(\xi), \ \  \nabla_{\xi}(\beta)= [\pr(\xi), \beta]+ D_{\rho(\beta)}(\xi).\]
Using these, each $\xi\in \Gamma(A)$ induces a Lie derivative $\Lie_{\xi}$ on $\Omega^1(M, A)$ by
\[ \Lie_{\xi}(\omega)(X)= \nabla_{\xi}(\omega(X))- \omega([\rho(\xi), X]),\]
and then, for the bracket of $\Jet^1A$, one finds
\[ [\xi, \eta]= j^1([\pr(\xi), \pr(\eta)])+ \Lie_{\xi}(D(\eta))- \Lie_{\eta}(D(\xi)).\]
\end{remark}



\begin{prop} Let $\G$ be a Lie groupoid and $A$ a Lie algebroid over $M$. Then:
\begin{enumerate}
\item[1.] The Cartan form $\theta_{\mathrm{can}}\in \Omega^1(\Jet^1\G, t^{\ast}A)$ is a multiplicative form with values in the adjoint representation.
\item[2.] The classical Spencer operator of $A$ (Example \ref{1-jets-convenient}), denoted here
\[ D^A: \Omega^1(\Jet^1A)\To \Omega^1(M, A),\] 
is a Spencer operator on the algebroid $\Jet^1A$ relative to $\pr: \Jet^1 A\to A$, where the induced action on $A$ is the adjoint action. 
\item[3.] If $A= Lie(\G)$, the Spencer operator of $\theta_{\mathrm{can}}$ (cf. Theorem \ref{t1}) is $D^A$.
\end{enumerate}
\end{prop}

\begin{proof} We first show that $\theta_{\mathrm{can}}$ is multiplicative, i.e. that:
\[(m^{\ast}\theta_{\mathrm{can}})|_{(\sigma_g,\sigma_h)} = \pr_1^{\ast}\theta_{\mathrm{can}} + \Ad_{\sigma_g}\pr_2^{\ast}\theta_{\mathrm{can}}.\]

We use the description from Remark \ref{when working with jets}. Let $\xi_1\in T_{\sigma_g}\Jet^1\G$ and $\xi_2\in T_{\sigma_h}\Jet^1\G$ be such that $\d s(\xi_1)=\d t(\xi_2)$. Denote by $X_1 = \d \pr(\xi_1) \in T_g\G$ and $v_1 = \d s(X_1) = \d s(\xi_1) \in T_{s(g)}\G$. Similarly, let $X_2 = \d \pr(\xi_2) \in T_h\G$ and $v_2 = \d s(X_2) = \d s(\xi_2)\in T_{s(h)}M$. Computing $\theta_{\mathrm{can}}(\d m (\xi_1, \xi_2))$ we find
\[\begin{aligned}
&\  R_{(gh)^{-1}}(\d \pr(\d m(\xi_1,\xi_2)) - (\sigma_g\cdot\sigma_h)(\d s (\d m(\xi_1,\xi_2))))= \\
&=R_{(gh)^{-1}}(\d m(X_1,X_2) - (\sigma_g\cdot\sigma_h)(v_2))\\
&=R_{(gh)^{-1}}(\d m(X_1,X_2) - \d m(\sigma_g(\lambda_{\sigma_2}(v_2)), \sigma_h(v_2)))\\
&=R_{(gh)^{-1}}(\d m(X_1-\sigma_g(\lambda_{\sigma_2}(v_2)), X_2-\sigma_h(v_2)))\\
&=R_{g^{-1}}(\d m(X_1-\sigma_g(\lambda_{\sigma_2}(v_2)),R_{h^{-1}}( X_2-\sigma_h(v_2))))\\
&=R_{g^{-1}}(\d m(X_1- \sigma_g(v_1), 0_{s(g)}) + \d m(\sigma_g(v_1)-\sigma_g(\lambda_{\sigma_2}(v_2)),R_{h^{-1}}( X_2-\sigma_h(v_2))))\\
&=R_{g^{-1}}(X_1- \sigma_g(v_1) + \d m(\sigma_g(v_1)-\sigma_g(\lambda_{\sigma_2}(v_2)),R_{h^{-1}}( X_2-\sigma_h(v_2))))\\
&=R_{g^{-1}}(X_1- \sigma_g(v_1)) + \Ad_{\sigma_g}(R_{h^{-1}}( X_2-\sigma_h(v_2)))\\
&=\theta_{\mathrm{can}}(\xi_1) + \Ad_{\sigma_g}\theta_{\mathrm{can}}(\xi_2)
\end{aligned}\] 
where we have used the fact that $\pr:\Jet^1\G\to \G$ is a Lie groupoid morphism.
Let $(D,l)$ denote the Spencer operator of $\theta_{\mathrm{can}}$. It is clear from the definition of $l$ that $l= \pr$ and it suffices to prove that $D$ satisfies the holonomicity condition 
$D(\jet^1\al)=0$, for all $\al \in \Gamma(A)$. Let $\zeta=\jet^1\al$. In the explicit formula \eqref{eq: explicit formula} for $D$, we remark that  $\phi_\zeta^\epsilon(x)=(\d\phi^\epsilon_\al)_x$, hence
\[\begin{aligned}
\theta_{\mathrm{can}}((\d\phi^\epsilon_\zeta)_x(X_x))&=R_{(\phi^\eps_\al(x))^{-1}}(\d \pr((\d\phi_{\zeta}^{\eps})_x(X_x)) - (\d \phi^{\eps}_\al)_x(\d s((\d\phi^\epsilon_\zeta)_x(X_x))))\\
&=R_{(\phi^\eps_\al(x))^{-1}}((\d \phi^{\eps}_\al)_x(X_x) - (\d \phi^{\eps}_\al)_x(X_x)) = 0,
\end{aligned}
\]
hence $D(\jet^1\al)(X)= 0$. 
\end{proof}

\begin{remark}\rm \
Similarly one talks about the $k$-jet groupoid $\Jet^k\G$ and the $k$-jet algebroid $\Jet^kA$; completely analogous one has a Cartan form 
\[ \theta_{\textrm{can}} \in \Omega^1(J^k\G, t^*\Jet^{k-1}A)\]
and the previous proposition holds for all $k$'s.
\end{remark}

\subsection{Multiplicative distributions; Theorem \ref{t1} $\Rightarrow$ Theorem \ref{t2}}
\label{sec: distributions}

We concentrate now on the case Pfaffian case ($k= 1$), explaining in particular that Theorem \ref{t2} follows from Theorem \ref{t1}. As we have already mentioned,  
while Cartan was working himself with $1$-forms, many people preferred the dual picture. 
Recall that, on any manifold $P$, one has a 1-1 correspondence between
\begin{enumerate}
\item regular $1$-forms $\theta\in \Omega^1(P, E)$, where $E$ is some vector bundle over $P$. 
\item distributions $\H$ on $P$, i.e. vector sub-bundles $\H\subset TP$.
\end{enumerate}
Regular means that $\theta$ is pointwise surjective. 
In one direction,  $\H= \textrm{Ker}(\theta)$; conversely, $E= TP/\H$ and $\theta$ is the canonical projection. This remark has a multiplicative version:

\begin{lemma} Let $\G$ be a Lie groupoid. Then for any representation $E$ of $\G$ and any regular $E$-valued multiplicative form $\theta\in \Omega^1(\G, t^{\ast}E)$
\[ \H_{\theta}:= \textrm{Ker}(\theta)\subset T\G\]
is a multiplicative distribution on $\G$. Moreover, any multiplicative distribution arises in this way.
\end{lemma}

The fact that $\H_{\theta}$ is multiplicative follows immediately. We now concentrate on the last part, which also gives us the opportunity for having a closer look at the multiplicativity condition for
distributions. Note that, given $\H\subset T\G$, multiplicativity of $\H$ (Definition \ref{def-pf-syst}) is equivalent to:
\begin{enumerate}
\item $\H$ is closed under $\d m$, i.e., for any $X_g\in \H_g$, $Y_h\in \H_h$ for which $\d s(X_g)=\d t(Y_h)$, $\d_{(g,h)}m(X_g,Y_h)\in \H_{gh}.$
\item $\H$ is closed under $\d i$, i.e., $\d i(\H_g)=\H_{g^{-1}}.$
\item At units $x=1_x$, $\H_x$ contains $T_xM$. 
\item $\H$ is $s$-transversal, i.e., $T\G=\H+T^s\G$.
\end{enumerate}
Note that the last condition is equivalent to the surjectivity of $(ds): \H\rmap TM$; it actually implies (using a dimension counting) that the last map is not only surjective, but also a submersion 
(which is necessary for $\H$ to be a Lie groupoid over $TM$). The same last condition implies that $\H^s=\H\cap T^s\G$ has constant rank. While
$T^s\G$ restricted to $M$ gives the Lie algebroid $A$ of $\G$ the restriction of $\H^s$ induces the symbol sub-bundle 
\[  \mathfrak{g}:= \H^s|_{M}\subset A. \]
Similarly, while right translations induce an isomorphism of vector bundles, $r: T^{s}\G \stackrel{\sim}{\To} t^{\ast}A$, $r(X_g)= R_{g^{-1}}(X_g)$,
it restricts to an isomorphism $\H^s\cong t^{\ast} \mathfrak{g}$. 
Passing to quotients, we obtain a vector bundle over $M$
\[ \ E:= A/\mathfrak{g} ,\]
and an isomorphism of vector bundles over $\G$ (where we use again $4.$ above)
\[ T\G/\H\simeq T^s\G/\H^s \overset{r}{\To} t^{\ast}(E) .\]
Hence the canonical projection $T\G\to T\G/\H$ can be interpreted as a form
\[ \theta_{\H}\in \Omega^1(\G, t^*E).\]
Finally, there is an induced ``adjoint action'' of $\G$ on $E$: for $g\in \G$, 
\[\Ad^\H_g: E_{s(g)}\To E_{t(g)} ,\ \ \Ad^\H_g(\al \mod \B) = (\Ad_{\sigma_g}\al)\mod \B,\]
where $\sigma_g: T_{s(g)}M \To \H_g\subset T_g\G$ is any splitting of $(\d s)_g$ 
and where $\Ad$ is the adjoint representation 
of $\Jet^1\G$ on $A$ (see Section \ref{sec: Jet groupoids}). With this:

\begin{lemma}\label{lemma-from-H-to-theta} $E$ is a representation of $\G$ and $\theta_{\H}\in \Omega^1(\G, E)$ is 
multiplicative.
\end{lemma}

\begin{proof} To see that $\Ad^\H$ is well defined we note that, if $\be \in \mathfrak{g}$, then
\[\Ad_{\sigma_g}\be = R_{g^{-1}}\d m (\sigma_g(\rho(\be)), \be),\]
which belongs to $\mathfrak{g}$ due to $\H$ being multiplicative. Moreover, if $\sigma'_g$ is another splitting of $\d s$ whose image lies in $\H$, then
\[\begin{aligned}
\Ad_{\sigma_g}\al - \Ad_{\sigma'_g}\al& =R_{g^{-1}}(\d m (\sigma_g(\rho(\al)), \al) - \d m (\sigma'_g(\rho(\al)), \al))\\
&= R_{g^{-1}}\d m (\sigma_g(\rho(\al)) - \sigma'_g(\rho(\al)), 0_{s(g)}) 
\end{aligned}\]
which also belongs to $\mathfrak{g}$, for all $\al \in \Gamma(A)$. It follows that $\Ad^{\H}_g$ is independent of the choice of splitting $\sigma_g$.

We now show that $\theta_{\H}$ is multiplicative for this representation. Observe that for $\xi \in T_g\G$, if $\tilde{\xi}$ is any lift of $\xi$ to $T_{\sigma_g}\Jet^1\G$, then
\[\theta_g(\xi) = \theta_{\mathrm{can},\sigma_g}(\tilde{\xi})\mod B,\]
where, again $\sigma_g$ is any splitting of $\d s$ whose image lies in $\H$, and $\theta_{\mathrm{can}}$ denotes the canonical form of $\Jet^1\G$ (see Section \ref{sec: Jet groupoids}). Also, since $\H$ is multiplicative, if $\sigma_g$ and $\sigma_h$ are splittings of $\d s$ whose image lie in $\H$, then also the image of $\sigma_g\cdot \sigma_h$ lies in $\H$ (whenever the product is defined). It follows that
\[\begin{aligned}
\theta_{gh}(\d m(\xi_1, \xi_2)) & = (\theta_{\mathrm{can}, \sigma_g\sigma_h}(\d m(\tilde{\xi_1},\tilde{\xi_2})))\mod B \\
&=(\theta_{\mathrm{can},\sigma_g}(\tilde{\xi_1}) + \Ad_{\sigma_g}\theta_{\mathrm{can},\sigma_h}(\tilde{\xi_2}))\mod B\\
&=\theta_g(\xi_1) + \Ad^\H_{\sigma_h}(\xi_2).
\end{aligned}\]
\end{proof}

Of course, Theorem \ref{t1} now follows from Theorem \ref{t2} applied to $\theta_{\H}$, combined with the 
reformulation of Spencer operators (Remark \ref{remark-dual}). What we still have to prove is that the explicit formula (\ref{eq: explicit formula}) for $D$ (from 
Theorem \ref{t1}) gives the explicit formula (\ref{D}) (from Theorem \ref{t2}). With the right hand side of (\ref{eq: explicit formula}) in mind, we consider more general expressions
of type:
\[(\Lie_{\al}\omega)_g:=\frac{d}{d\eps}\big{|}_{\eps=0}(\varphi^{\eps}_{\al^r}(g))^{-1}\cdot(\varphi^\epsilon_{\al^r})^*\omega|_{\varphi^\epsilon_{\al^r}(g)}\]
for $\alpha\in \Gamma(A)$ and $\omega\in \Omega^k(\G,t^*E)$ (see also Remark \ref{right-invariance}). This defines
\[\Lie_{\al}:\Omega^k(\G,t^*E)\To \Omega^k(\G,s^*E).\]

\begin{lemma} For any vector field $\xi\in \X(\G)$, $\omega \in \Omega^k(\G, t^{\ast}E)$, $g\in \G$:
\[[i_\xi,\Lie_{\al}](\omega)_g=g^{-1}\cdot\omega_g([\xi,\al^r]),\]
\end{lemma}

Note that this implies (\ref{D}). Indeed, if $\tilde{X}$ is as in the statement, using the lemma for $g= 1_x= x$, $\omega= \theta$, since $D(\alpha)(x)= \Lie_{\al}(\theta)(x)$ and $\theta(\tilde{X}_x)= 0$, 
\[ D_{X}(\al)(x)=[i_{\tilde{X}}, \Lie_{\alpha}](\theta)_x= \theta_{x}([\tilde X,\al^r])= \text{[}\tilde{X},\al^r]_{x} \mod \mathfrak{g}.\]

\begin{proof}[Proof of the lemma]
We apply the chain rule to the composition 
\[\frac{d}{d\eps}\big{|}_{\eps=0}(\varphi^{-\eps}_{\al^r}(g))^{-1}\cdot\omega_{\varphi^{-\eps}_{\al^r}(g)}(\d_g\varphi^{-\eps}_{\al^r}(\xi_g))= f_1\circ f_2,\]
where
\[f_1: I\times s^{-1}(s(g)) \To E_{s(g)}, \qquad f_1(\eps,h)= h^{-1}\cdot\omega_h(\d_{\varphi^{\eps}_{\al^r}(h)}\varphi^{-\eps}_{\al^r}(\xi_{\varphi^{\eps}_{\al^r}(h)})), \]
and
\[f_2: I \To s^{-1}(s(g)), \qquad f_2(\eps) = \varphi^{-\eps}_{\al^r}(g).\]
We obtain,
\begin{equation*}
\begin{aligned}
\frac{d}{d\eps}&\big{|}_{\eps=0}(\varphi^{-\eps}_{\al^r}(g))^{-1}\cdot\omega_{\varphi^{-\eps}_{\al^r}(g)}(\d_g\varphi^{-\eps}_{\al^r}(\xi_g))=\\
&=\frac{d}{d\epsilon}\big{|}_{\eps=0}g^{-1}\cdot\omega_{g}(\d_{\varphi^{\eps}_{\al^r}(g)}\varphi^{-\eps}_{\al^r}(\xi_{\varphi^{\eps}_{\al^r}(g)}))+
\frac{d}{d\eps}\big{|}_{\eps=0}(\varphi^{-\eps}_{\al^r}(g))^{-1}\cdot\omega_{\varphi^{-\eps}_{\al^r}(g)}(\xi_{\varphi^{-\eps}_{\al^r}(g)}),
\end{aligned}
\end{equation*}
or in other words,
\begin{equation*}
-(\Lie_{\al}\omega)(\xi)(g)=g^{-1}\cdot\omega([\al^r,\xi])-\Lie_{\al}(\omega(\xi))(g).
\end{equation*}
i.e. the equation in the statement.
\end{proof}

\subsection{Cartan connections on groupoids}
\label{Cartan connections on groupoids}

The notion of Cartan connections on a Lie groupoid $\G$ arises when looking at the adjoint representation of $\G$ \cite{Abadrep}. It is straightforward to see that the
definition from {\it loc.cit} is equivalent to:

\begin{definition} A Cartan connection on a Lie groupoid $\G$ over $M$ is a multiplicative distribution $\H\subset T\G$
which is complementary to $\textrm{Ker}(\d s)$.
\end{definition}

As for any Ehresmann connection, we will denote the inverse of $(\d s)|_{\H}$ by 
\[ \textrm{hor}: TM\To \H\subset T\G.\]


On the infinitesimal side, we deal with classical connections
\[ \nabla: \X(M)\times \Gamma(A) \To \Gamma(A) \]
on the vector bundle underlying a Lie algebroid $A$. For such a connection, one has the notion of basic curvature
\[ R_{\nabla}^{\textrm{bas}}\in \Omega^2(M, \textrm{Hom}(TM, A))\]
which has appeared in the literature in various contexts (e.g. in \cite{homotopy, Blaom}): 
\[ R_{\nabla}^{\textrm{bas}}(\alpha,\beta)(X):=\nabla_X([\alpha,\beta])-[\nabla_X(\alpha),\beta]-[\alpha,\nabla_X(\beta)]-\nabla_{\nabla_{\beta}^{\textrm{bas}}X}(\alpha)
+\nabla_{\nabla_{\alpha}^{\textrm{bas}}X}(\beta),\]
where $\alpha, \beta $ are sections of $A$ and $X,Y$ are vector
fields on $M$ and
\[\nabla^{\textrm{bas}}_{\alpha}(X)= \rho(\nabla_{X}(\alpha))+ [\rho(\alpha), X].\]

\begin{definition} A Cartan connection on a Lie algebroid $A$ is any connection $\nabla$ on the vector bundle $A$ whose
basic curvature $R_{\nabla}^{\textrm{bas}}$ vanishes.
\end{definition}

Such pairs $(A, \nabla)$ are the Cartan algebroids of \cite{Blaom}. Theorems \ref{t2} and \ref{t3} give:

\begin{theorem} For any Cartan connection $\H$ on a Lie groupoid $\G$ over $M$, 
\begin{eqnarray}\label{nabla-basic}
\nabla: \X(M)\times \Gamma(A) \To \Gamma(A), \ \nabla_X\alpha(x)= \d s([\mathrm{hor}(X),\alpha^r]_x) 
\end{eqnarray}
is a Cartan connection on the Lie algebroid $A$ of $\G$.

When $\G$ is $s$-simply connected, this gives a bijection between Cartan connections $\H$ on $\G$
and Cartan connections $\nabla$ on the algebroid $A$. Moreover, $\H$ is involutive if and only if $\nabla$ is flat. 
\end{theorem}

Under topological condition, the existence of flat Cartan connections implies that the groupoid must come from the action of a Lie group. For instance:

\begin{corol}\label{flat-Cartan}
If $\G$ is an $s$-simply connected Lie groupoid over a compact $1$-connected manifold $M$ and if $\G$ admits a 
flat Cartan connection $\H$, then $\G$ is isomorphic to an action Lie groupoid $H\ltimes M$ associated to
a Lie group $H$ acting on $M$ (as defined in Example \ref{ex-flat-Cartan}).
\end{corol}

\begin{proof} The flatness of the associated $\nabla$ and the $1$-connectedness of $M$ implies that $A$ is a trivial
bundle: $A= \mathfrak{h}\times M$ for some vector space $\mathfrak{h}$ and the constant sections correspond to flat sections. 
The vanishing of the basic curvature implies that the bracket of constant sections is again constant; hence one has an 
induced Lie algebra structure on $\mathfrak{h}$; the anchor of $A$ becomes an infinitesimal action. Due to the
compactness of $M$, one can integrate this action to an action of the $1$-connected Lie group $H$ whose Lie algebra
is $\mathfrak{h}$. Then $H\ltimes M$ and $\G$ are two Lie groupoids with $1$-connected $s$-fibers and with the same Lie algebroid; hence
they are isomorphic. 
\end{proof}

\subsection{1-cocycles and the van Est map}
\label{1-cocyles and relations to the van Est map}

Another interesting case of the main theorem is $k= 0$, when we recover the van Est map relating differentiable and algebroid cohomology \cite{WeinsteinXu, Crainic:Van}, in degree $1$. Since this case will be used later on and also in order to fix the terminology, we discuss it separately here. In particular, we will provide a simple direct argument, based on  Lie's II theorem \cite{MoerdijkMrcun, Mackenzie:General} which says that, if $\G$ is $s$-simply connected and $\H$ is any Lie groupoid , then for any Lie algebroid morphism $\varphi: A(\G) \to A(\H)$, there exists a unique  Lie groupoid morphisms $\Phi: \G \to \H$ such that $\varphi = \d \Phi |_{A(\G)}$.

Let $\G$ be a Lie groupoid over $M$, and $E$ a representation of $\G$. We will denote by $\G^{(p)}$ the space of strings of $p$ composable arrows on $\G$, and by $t: \G^{(p)} \to M$ the map which associates to $(g_1, \ldots, g_p)$ the point $t(g_1)$. 
A differentiable $p$-cochain on $\G$ with values in $E$ is a smooth section $c: \G^{(p)} \to t^{\ast}E$. We denote the space of all such cochains by  $C^p(\G,E)$. The differential $\delta: C^p(\G,E) \to C^{p+1}(\G,E)$ is
\[\begin{aligned}
\delta c (g_1, \ldots, g_{p+1}&) = g_1c(g_2, \ldots, g_{p+1}) +\\
&+\sum_{i=1}^p(-1)^i c(g_1, \ldots , g_ig_{i+1}, \ldots ,g_{p+1}) +(-1)^{p+1}c(g_1, \ldots, c_p),
\end{aligned}\]
For each $p$ we consider the space of $p$-cocycles
\[Z^p(G,E) = \ker (\delta: C^p(\G,E) \to C^{p+1}(\G,E)).\]
We recognize the multiplicative $c\in \Omega^0(\G, E)$ as the elements of $Z^1(G,E)$. 

At the infinitesimal side, given a representation $\nabla$ of $A$ on $E$, one defines the de Rham cohomology of $A$ with coefficients in $E$ as the cohomology of the complex $(C^{\ast}(A,E), d)$, where $C^{p}(A,E) = \wedge^pA^{\ast}\otimes E$ and 
\[\begin{aligned}
d\omega(\al_0, \ldots, \al_p) =  \sum_{i}&(-1)^i\nabla_{\al_i}\omega(\al_0, \ldots, \widehat{\al_i},\ldots, \al_p) +\\
&+\sum_{i<j}(-1)^{i+j}\omega([\al_i, \al_j], \al_0, \ldots, \widehat{\al_i}, \ldots, \widehat{\al_j}, \ldots, \al_p).
\end{aligned}\]
As above, the space of $p$-cocycles is denoted by $Z^p(A ,E)$ and 
we recognize the $E$-valued $0$-Spencer operators as $1$-cocycles on $\A$ with values in $E$.

The van Est map is a map of cochain complexes
\[\ve: C^{\ast}(\G, E) \To C^{\ast}(A,E),\]
which induces isomorphism in cohomology under certain connectedness conditions on the $s$-fibers \cite{WeinsteinXu, Crainic:Van}. We will concentrate on the degree $1$ cochains. In this case, for $c\in C^{1}(\G, E)$, $\ve(c)$ is given by: 
\[\vartheta_x(c)(\al) = \frac{\d}{\d \eps}\big{|}_{\eps = 0}g_{\eps}^{-1}c(g_{\eps}),\]
where $\al \in A_x$, and $g_{\eps}$ is any curve in $s^{-1}(x)$ such that $g_0 = 1_x$, $\frac{\d}{\d \eps}\big{|}_{\eps = 0}g_{\eps} = \al$. 
Using $g_{\epsilon}= \phi^{\epsilon}_{\alpha}(x)$, we recognize our formula for the Spencer operator from Theorem \ref{t1}. Hence our main theorem gives:


\begin{prop}\label{prop: Van Est}
If $\G$ is $s$-simply connected, then the van Est map induces an isomorphism
\[\ve: Z^1(\G, E) \To Z^1(A,E),\]
where $Z^1(A,E)$ denotes the closed elements of $C^1(A,E)$.
\end{prop}

\begin{proof}
Consider the semi-direct product $\G\ltimes E \tto M$, whose space of arrows consists of pairs 
$(g,v) \in \G \times E$ with $t(g) = \pi(v)$ and 
\[ s(g,v) = s(g), \quad t(g,v) = t(g) = \pi(v), \quad (g,v)(h,w) = (gh, v+gw) .\] 
The fact that $c \in Z^1(\G,E)$ satisfies the cocycle condition is equivalent to 
\[ \tilde{c}:= (\textrm{Id}, c): \G\To \G\ltimes E\]
being a morphism of groupoids. The Lie algebroid of $\G\ltimes E$ is $A\ltimes E = A\oplus E$ with 
\[\rho(\al,s) = \rho(\al), \text{ and } [(\al,s),(\al',s')] = ([\al,\al'], \nabla_{\al}s' - \nabla_{\al'}s).\]
As before, $\omega \in C^1(A,E)$ is a cocycle if and only if  $\tilde{\omega}:= (\textrm{Id}, \omega) : A\To A \ltimes E$
is a Lie algebroid morphism. It is easy to see that, for $c \in Z^1(\G,E)$, the Lie functor applied to $\tilde{c}$ is $\widetilde{\ve(c)}$ hence the result follows from Lie II.
\end{proof}


\begin{remark}\label{along-s}\rm \ Sometimes it is more natural to consider cochains along $s$, i.e. sections of the pull-back of $E$ via 
$s: \G^{(p)} \to M$, $(g_1, \ldots, g_p)\mapsto s(g_p)$. In degree $1$, one deals with $c\in \Gamma(\G, s^{\ast}E)$ and the 
cocycle condition becomes
\[ c(gh)= h^{-1}\cdot c(g)+ c(h) .\]
Of course, one can just pass to cocycles in the previous sense by considering 
\[ \overline{c}(g)= g^{-1}\cdot c(g).\]
Note that the associated algebroid cocycle is simply $\ve(\overline{c})(\alpha_x)= (\d c)_x(\alpha_x)$.
\end{remark}

\begin{remark}[cocycles as representations]\label{remark-cocycles}\rm 
Yet another interpretation of $1$-cocycles is obtained when $E= \mathbb{R}$ is the trivial representation of $\G$ (and of $A$).
On the groupoid side, any 1-cocycle $c\in C^{1}(\G)$ induces a representation, denoted $\mathbb{R}_c$, of $\G$ on the trivial line bundle:
\[ g\cdot t:= e^{c(g)}t .\]
When the $s$-fibers of $\G$ are connected, it is not difficult to see that this gives a 1-1 correspondence. Similarly, one has a 
1-1 correspondence between $1$-cocycles on $A$ and structures of representations of $A$ on the trivial line bundle; given $a\in Z^1(A)$,
the corresponding representation, denoted $\mathbb{R}_{a}$,  is determined by 
\[ \nabla_{\al}(1)= a(\al) .\]

It is clear that, in this case, the van Est map (and Proposition \ref{prop: Van Est}) becomes the Lie functor between representations of $\G$ and those of $A$.
\end{remark}

For later use, we give the following: 

\begin{prop}\label{prop: appendix}
Let $v \in T_g\G$ be a vector tangent to the $t$-fiber of $\G$, and let $c \in Z^1(\G,E)$ be a cocycle. Then
\[(\d c)_g(v) = g\cdot (\d c)_{s(g)}(\d_gL_{g^{-1}}v).\]
\end{prop}

\begin{proof}
Let $\gamma: I \to \G$ be a curve in the $t$-fiber through $g$ whose velocity at $\eps = 0$ is $v$. Since $c$ is a cocycle, it follows that
\[c(g^{-1}\cdot\gamma(\eps)) = g^{-1}\cdot c(\gamma(\eps)) + c(g^{-1}),\]
for all $\eps \in I$. Finally, one differentiates w.r.t. $\eps$.
\end{proof}

\subsection{Trivial Coefficients}
\label{sec: trivial coefficients}

Another interesting case of the main theorem is when $E$ is the trivial representation. This case was well studied
because of its relevance to Poisson geometry. Our Theorem \ref{t1} recovers the most general results in this context. 
The key remark is that, when $E= \mathbb{R}$ with the trivial action, any bundle map $l:A \to \wedge^{k-1}T^*M$ is canonically the symbol of an $E$-Spencer operator, namely $\al\mapsto d(l(\al))$. We see that any $E$-Spencer operator can be decomposed as
\[ D(\al)= \nu(\al)+ d(l(\al))\]
where, this time, $\nu$ is tensorial. Rewriting everything in terms of $\nu$ and $\al$ we obtain the result of \cite{CrainicArias, BursztynCabrera}:

\begin{prop}\label{prop: trivial coefficients}
Let $\G$ be an $s$-simply connected Lie groupoid with Lie algebroid $A$. Then there is a one-to-one correspondence between multiplicative forms $\theta \in \Omega^k(\G)$ and pairs $(\nu, l)$ consisting of vector bundle maps 
\begin{equation}\label{1}
\left\{\begin{aligned}
\nu:A&\To \wedge^kT^*M\\
l:A&\To \wedge^{k-1}T^*M
\end{aligned}\right.
\end{equation}
satisfying 
\begin{equation}\label{eq: compatibility-tc}\left\{
\begin{aligned}
\nu([\al,\be] ) &= \Lie_{\rho(\al)}\nu(\be) - i_{\rho(\be)}d\nu(\al)\\
i_{\rho(\be)}\nu(\al) &= \Lie_{\rho(\al)}l(\be)  - i_{\rho(\be)}dl(\al)- l([\al,\be])\\
i_{\rho(\al)}l(\be) &= -i_{\rho(\be)}l(\al),
\end{aligned}\right.
\end{equation}
for all $\al,\be \in \Gamma(A)$. The correspondence $\theta\mapsto (\nu_{\theta}, l_{\theta})$ is given explicitly by
\[  \nu_{\theta}(\al)=u^{\ast}(i_{\al}d\theta), \ \  l_{\theta}(\al)=u^{\ast}(i_\al\theta).\]
\end{prop}


For $\phi\in \Omega^{k+1}(M)$,
the cohomologically trivial form $\delta(\phi)= s^{\ast}\phi- t^{\ast}\phi$ gives (see Examples \ref{ex: form on base} and \ref{ex: form on base-2})
\[ \nu_{\delta(\phi)}(\al)=  -i_{\rho(\al)}(d\phi) ,\qquad   l_{\delta(\phi)}(\phi)= -i_{\rho(\al)}(\phi),\]
hence one obtains an infinitesimal characterization for cohomological triviality.

Note that in the case of trivial coefficients it makes sense to talk about the DeRham differential of a multiplicative form (itself multiplicative).
From the last formulas in the proposition, we have:
\[ \nu_{d\theta}= 0,\qquad l_{d\theta}= \nu_{\theta},\]
hence one immediately obtains infinitesimal characterizations of multiplicative forms which are closed. More generally, given $\phi\in \Omega^{k+1}(M)$
closed, one says that $\theta\in \Omega^k(\G)$ is $\phi$-closed if $d\theta= s^{\ast}\phi- t^{\ast}\phi$. One obtains for instance the following, which when $k= 2$
gives the main result of \cite{BCWZ}.

\begin{corol}\label{cor-B-field}
Assume that $\G$ is $s$-simply connected, and that $\phi\in \Omega^{k+1}(M)$ is closed. Then there is a bijection between $\phi$-closed, multiplicative $\theta \in \Omega^k(\G)$ and
\[l: A \To \Lambda^{k-1}T^{\ast}M\]
which are vector bundle maps satisfying
\[\left\{\begin{aligned}
l([\al,\be]) &= \Lie_{\rho(\al)}l(\be) - i_{\rho(\be)}\d l (\al)+ i_{\rho(\al)\wedge \rho(\be)}(\phi) \\
i_{\rho(\al)} l(\be) &= -i_{\rho(\be)}l(\al).
\end{aligned}\right.\]
\end{corol}

\begin{example}\rm \  To make the discussion on Jacobi manifolds of the next section more transparent, we briefly recall here the Poisson case. 
Let $(M, \pi)$ be a Poisson manifold and let $A= T^{\ast}M$ be endowed with the induced algebroid structure (see e.g. \cite{Zung}),
which is assumed to come from an $s$-simply connected Lie groupoid $\Sigma$. 
Then previous proposition can be applied to $k= 2$, $\phi= 0$ and $l= \textrm{Id}_{T^{\ast}M}$; this gives rise to $\omega\in \Omega^2(\Sigma)$ 
closed and multiplicative, and also non-degenerate (because $l$ is an isomorphism), making $\Sigma$ into a symplectic groupoid. 
\end{example}

\subsection{Contact groupoids}
\label{Contact groupoids}

In analogy with symplectic groupoids and Poisson geometry, contact groupoids are the global counterpart of Jacobi manifolds. Although this has been known 
for a while (see e.g. \cite{Jacobi} and the references therein), the existing approaches have been
rather in-direct (by using ``Poissonization'', applying the similar results from Poisson geometry, then passing to quotients). What happens is that contact groupoids
require the use of non-trivial coefficients; therefore, our main theorem now allows for a direct approach.  Furthermore, using the slightly 
more general setting of Kirillov's local Lie algebras, the approach becomes less computational and more conceptual.


The difference between Jacobi manifolds and local Lie algebras, or the difference between their global counterparts, is completely analogous the
difference between the two related but non-equivalent notions of contact manifolds that one finds in the literature. Here we follow the terminology of \cite{contact}.

\begin{definition} Let $M$ be a manifold. 
\begin{itemize}
\item A contact structure on $M$ is a contact form $\theta$, i.e. a regular $1$-form $\theta\in \Omega^1(M)$ with the property that the restriction of $d\theta$ to the distribution $\H_{\theta}= \textrm{Ker}(\theta)$ is
pointwise non-degenerate. 

\item A contact structure in the wide sense on $M$ is a contact hyperfield, i.e. a codimension one distribution $\H\subset TM$ 
which is maximally non-integrable. 
\end{itemize}
\end{definition}

Here maximal non-integrability can be understood globally as follows. First, $\H$ induces a line bundle
\[ L= TM/\H .\]
Then, the Lie bracket modulo $\H$ induces a ($C^{\infty}(M)$-)bilinear map
\begin{equation}\label{I-H} 
I: \H\times \H \To L, \ \ (X, Y)\mapsto [X, Y]\ \textrm{mod}\ \H
\end{equation}
and the maximal non-integrability of $\H$ means that $I$ is non-degenerate. The contact case is obtained when $L$ is the trivial line bundle. 
Passing to groupoids:

\begin{definition} Let $\Sigma$ be a Lie groupoid over $M$. 
\begin{itemize}
\item A contact structure on the groupoid $\Sigma$ is a pair $(\theta, r)$ consisting of a smooth map $r: \Sigma\to \mathbb{R}$ (the Reeb cocycle) and a contact form $\theta\in \Omega^1(\Sigma)$
which is $r$-multiplicative in the sense that 
\[ m^{\ast}\theta= \pr_{2}^{\ast}(e^{-r})\pr_{1}^{\ast}\theta+ \pr_{2}^{\ast}\theta .\]

\item A contact structure in the wide sense on the groupoid is a contact hyperfield $\H$ on $\Sigma$ which is multiplicative.
\end{itemize}
\end{definition}

Regarding the first notion note that the equation above implies that, indeed, $r$ is a $1$-cocycle; hence it induces a representation $\mathbb{R}_r$ of $\Sigma$ (cf. Remark \ref{remark-cocycles}). Also, one has the following immediate but important remark,
which will allow us to reconstruct $\theta$ from associated Spencer operators:

\begin{lemma} 
The $r$-multiplicativity of $\theta$ is equivalent to the fact that $e^{r}\theta\in \Omega^1(\Sigma)$ is multiplicative as a form with values in the representation $\mathbb{R}_r$.
\end{lemma} 

The following indicates the conceptual advantage of the ``wide'' point of view.

\begin{lemma} \label{lemma-contact-versus-wide} Assume that the $s$-fibers of $\Sigma$ are connected. Then the construction $(\Sigma, \theta, r)\mapsto (\Sigma, \textrm{Ker}(\theta))$ induces a 1-1 correspondence between contact groupoids $(\Sigma, \theta, r)$ and contact groupoids in the wide sense $(\Sigma, \H)$ with the property that the associated line bundle is trivial.
\end{lemma}

\begin{proof} It is clear that $\textrm{Ker}(\theta)$ has the desired properties. Conversely, assume that  we start with $(\Sigma, \H)$ so that $L$ is trivial. 
First of all, we know that the multiplicativity of $\H$ makes $L$ into a representation
of $\Sigma$ (cf. subsection \ref{sec: distributions}); we also know that a representation of $\Sigma$ on a trivial line bundle is uniquely determined by a $1$-cocycle (cf. e.g. Remark \ref{remark-cocycles}); this
gives rise to the cocycle $r$. Then the canonical projection $T\Sigma \to L$ gives a $1$-form $\overline{\theta}\in \Omega^1(\Sigma)$ which, by Proposition \ref{lemma-from-H-to-theta}  is multiplicative as a 
form with coefficients in $\mathbb{R}_r$. Hence, by the previous lemma, $\theta:= e^{-r}\overline{\theta}$ is $r$-multiplicative.
\end{proof}\\

We now pass to the corresponding infinitesimal structures.

\begin{definition} Let $M$ be a manifold.
\begin{itemize}
\item A Jacobi structure on $M$ is a pair $(\Lambda, R)$ consisting of a 
bivector $\Lambda$ and a vector field $R$ (the Reeb vector field)
satisfying 
\[ [\Lambda, \Lambda]= 2 R\wedge \Lambda,\ \ [\Lambda, R]= 0.\]

\item A Jacobi structure in the wide sense on $M$ is a pair $(L, \{\cdot, \cdot \})$ consisting of a line bundle $L$ over $M$ and a Lie
bracket $\{\cdot, \cdot \}$ on the space of sections $\Gamma(L)$, with the property that it is local, i.e. 
\[ \textrm{sup}(\{u, v\})\subset \textrm{sup}(u)\cap \textrm{sup}(v)\ \ \ \ \forall\ u, v\in \Gamma(L).\]
\end{itemize}
\end{definition}

The second notion appears in the literature under various names. Kirillov introduced them under the notion of local Lie algebra \cite{Kirillov}; Marle uses the term Jacobi bundle \cite{Marle}. Our term ``wide'' is ad-hoc, for compatibility with the previous definitions; however, we will also say that $L$ is a Jacobi bundle. 

For a Jacobi bundle $L$, Kirillov 
proves that $\{\cdot, \cdot\}$ must be a differential operator of order at most one in each argument. When $L= \mathbb{R}_M$ is the trivial line bundle, this implies that the
bracket must be of type 
\[  \{f, g\}_{\Lambda, R}= \Lambda(df, dg)+ \Lie_R(f)g- f\Lie_R(g)\ \ \ (f, g\in \Gamma(\mathbb{R}_M)= C^{\infty}(M))\]
for some bivector $\Lambda$ and vector field $R$. A straightforward check shows that this satisfies the Jacobi identity if and only if $(\Lambda, R)$ 
is a Jacobi structure. Hence, one obtains the following well-known:

\begin{lemma} \label{Jacobi-wide}
$(\Lambda, R)\mapsto (\mathbb{R}_{M}, \{\cdot, \cdot\}_{\Lambda, R})$ defines a bijection between Jacobi structures and local Lie algebras with trivial underlying line bundle. 
\end{lemma}

Next, we sketch the connection between contact groupoids and Jacobi structures pointing out the relevance of the Spencer operator and of Theorem \ref{t1}.

\subsubsection*{The Lie functor}
In one direction (the Lie functor), starting with a contact groupoid in the wide sense $(\Sigma, \H)$,
there is an induced Jacobi bundle on $M$. The relevance of the Spencer operator $D$ associated to $\H$ is the following: since $\H$ is contact, it follows that
the vector bundle map associated to $D$ (cf. Example \ref{1-jets-convenient}),
\[ j_{D}: A\rmap \Jet^1L,\]
is an isomorphism, where $A$ is the Lie algebroid of $\Sigma$ and $L$ is the line bundle associated to $\H$. Identifying $A$ with $\Jet^1L$, we
obtain a Lie bracket $[\cdot, \cdot]$ on $\Jet^1L$. On $\Gamma(L)$ we define the bracket
\[ \{u, v\}:= \pr ([j^1u, j^1v]).\]

\begin{lemma} $(L, \{\cdot, \cdot\})$ is a Jacobi structure in the wide sense.
\end{lemma}

\begin{proof} The bracket is clearly local, hence we are left with proving the Jacobi identity. 
For this it suffices to show that
\[ [j^1u, j^1v]= j^1\{u, v\}\]
for all $u, v\in \Gamma(L)$. Note that, after the identification of $A$ with $\Jet^1L$, the $D$ is identified with the
classical Spencer operator (see Example \ref{1-jets-convenient}); in particular, $D(\xi)= 0$ if and only if $\xi$ is the first jet of a section. 
Fixing $u$ and $v$, the equation (\ref{eq: compatibility-1}) for the Spencer operator implies that $D$ kills $[j^1u, j^1v]$,
hence $[j^1u, j^1v]= j^1s$ for some $s$. Applying $\pr$, we find $s= \{u, v\}$.
\end{proof}\\

Of course, starting with a contact groupoid $(\Sigma, \theta, r)$, Lemma \ref{lemma-contact-versus-wide} and Lemma \ref{Jacobi-wide}
ensure the existence a Jacobi structure $(\Lambda, R)$ on the base.

\subsubsection*{Integrability}
Conversely, start with a Jacobi structure in the wide sense $(L, \{\cdot, \cdot\})$ on $M$. With the Lie functor in mind, the strategy is quite clear:
consider the induced Lie algebroid structure on $\Jet^1L$ with the property that
\[ [j^1u, j^1v]= j^1\{u, v\}\]
for all $u, v\in \Gamma(L)$ and show that the classical Spencer operator $D$ is a Spencer operator with respect to this Lie algebroid structure. Then,
if $\Jet^1L$ comes from a Lie groupoid $\Sigma$, assumed to be $s$-simply connected, integrating $D$ gives the multiplicative hyperfield $\H$ on $\Sigma$ and
the fact that $j_{D}$ is an isomorphism implies that $\H$ is contact. 

For instance, when $(L, \{\cdot, \cdot\})$ comes from a Jacobi structure $(\Lambda, R)$, $\Jet^1L= T^*M\oplus \mathbb{R}$ and, starting from the previous formula, one
finds the Lie algebroid
\[ A_{\Lambda, R}:= T^{\ast}M\oplus \mathbb{R},\]
with anchor $\rho_{\Lambda, R}= \rho$ given by
\[ \rho(\eta, \lambda)= \Lambda^{\sharp}(\eta)+ \lambda R\]
and bracket 
\begin{align}
& [(\eta, 0), (\xi, 0)]_{\Lambda, R}= ([\eta, \xi]_{\Lambda}- i_{R}(\eta\wedge \xi), \Lambda(\eta,\xi)),   \nonumber\\
& [(0, 1), (\xi, 0)]_{\Lambda, R}= (\Lie_{R}(\xi), 0), \nonumber  \\
& [(0, 1), (0, 1)]_{\Lambda, R}= (0, 0)\nonumber
\end{align}
(extended to general elements using bilinearity and the Leibniz identity). The classical Spencer operator becomes 
\[ D: \Gamma(A_{\Lambda, R})\To \Omega^1(M), \ \ D(\eta, f)= \eta+ \d f,\ l(\eta, f)= f  .\]
Of course, checking directly that $D$ is a Spencer operator on $A_{\Lambda, R}$ is rather tedious. The advantage of the ``wide'' point of view is that 
it provides a more compact and computationally free approach. 

So, let's return to our $(L, \{\cdot, \cdot\})$. It is rather unfortunate (and surprising) that the 
definition of the associated Lie algebroid $\Jet^1L$ is missing from the literature. The remaining part of this section is mostly devoted to this point (after that, the 
part with the Spencer operator is immediate). The starting point is the result of Kirillov mentioned above: $\{u, v\}$ must be a differential operator of order at most
one in each argument. To fix the notations, recall that a differential operator 
\[ P: \Gamma(E)\To \Gamma(F)\]
of order at most one, where $E$ and $F$ are vector bundles, has a symbol
\[ \sigma_{P}\in \Gamma(TM\otimes \textrm{Hom}(E, F)).\]
The defining property is
\[ P(fs)= fP(s) + \sigma_{P}(df)(s) \]
for all $s\in \Gamma(E)$, $f\in C^{\infty}(M)$. Of course, when $E= F= L$ is one-dimensional, we get $\sigma_{P}\in \Omega^1(M)$. 
Fixing $u\in \Gamma(L)$, applying this to the operator $\{u, \cdot\}$, we denote the associated symbol by $\rho^1(u)$. This defines a map
\[ \rho^1: \Gamma(L)\To \X(M),\]
characterised by the property that
\[ \{u, fv\}= f\{u, v\}+ \Lie_{\rho^1(u)}(f) v\]
for all $u, v\in \Gamma(L)$. A straightforward computation with the Jacobi identity for $u, fv, w$ combined with the last equations implies that $\rho^1$ is a Lie algebra map:
\[ \rho^1(\{u, v\})= [\rho^1(u), \rho^1(v)].\]
 Unlike the case of Lie algebroids, $\rho^1$ need not be $C^{\infty}(M)$-linear. However, it is a differential operator of order at most one; hence it
satisfies the equation 
\[ \rho^1(fu)= f\rho^1(u)+ \rho^2(df\otimes u) ,\]
where $\rho^2$ is the symbol of $\rho^1$, interpreted as a vector bundle map
\[ \rho^2: \textrm{Hom}(TM, L)\To TM .\]
We define the anchor of $\Jet^1L$ by putting $\rho^1$ and $\rho^2$ together (see (\ref{J-decomposition})):
\[ \rho: \Gamma(\Jet^1L)\cong \Gamma(L)\oplus \Omega^1(M, L)\stackrel{\rho^1-\rho^2}{\To} \X(M)\]
Using the classical Spencer operator, one can write more compactly:
\[ \rho(\xi)= \rho^1(\pr(\xi))- \rho^2(D^{\textrm{clas}}(\xi)).\]
Finally, the Lie bracket for $\Jet^1L$ is, as we wanted, given by
\[ [j^1u, j^1v]:= j^1(\{u, v\}),\]
extended by the Leibniz identity to arbitrary sections (see also Remark \ref{curiosity}).

\begin{lemma} 
$(\Jet^1L, [\cdot, \cdot],\rho)$ is a Lie algebroid.
\end{lemma}

\begin{proof} The Leibniz identity holds by construction. For the Jacobi identity $\textrm{Jac}(\xi_1, \xi_2, \xi_3)= 0$, 
it is clearly satisfied when the $\xi_i$'s are first jets of sections of $L$. Hence it suffices to remark that the expression 
$\textrm{Jac}$ is $C^{\infty}(M)$-linear in all arguments. Using the Leibniz identity, we see that this is equivalent to the fact that the
anchor is a Lie algebra map:
\[ \rho([\xi_1, \xi_2])= [\rho(\xi_1), \rho(\xi_2)].\]
This time, the Leibniz identity implies that the difference  between the two terms is $C^{\infty}(M)$-bilinear, hence it suffices to
check it when $\xi_1= j^1u$, $\xi_2= j^1v$. This is equivalent to $\rho^1$ being a Lie algebra map.
\end{proof}\\

\begin{lemma} The classical Spencer operator $D^{\textrm{clas}}: \Gamma(\Jet^1L) \To \Omega^1(M, L)$ 
is a Spencer operator on the Lie algebroid $\Jet^1L$. 
\end{lemma}

Of course, the action $\nabla$ of $\Jet^1L$ on $L$ is the one induced by the formula (\ref{nabla-out-of-D}); hence it is characterised by
\[ \nabla_{j^1u}(v)= \{u, v\}. \]

\begin{proof} First note that the equation 
(\ref {eq: compatibility-2}) is satisfied (it is $C^{\infty}(M)$-linear in the arguments and, on holonomic sections, it reduces to the previous formula
for $\nabla$). In turn, this implies that the formula \ref{eq: compatibility-1} is $C^{\infty}(M)$-linear in the arguments hence, again, it suffices
to check it on holonomic sections, when it becomes $0= 0$. 
\end{proof}

\begin{remark}\label{curiosity}\rm As a curiosity, note that $\nabla$ and $[\cdot, \cdot]$ can be written on general elements using the
Spencer operator $D= D^{\textrm{clas}}$ as:
\[ \nabla_{\xi}(v)= \{\pr(\xi), v\}+ D(\xi)(\rho^1(v)),\]
\[ [\xi, \eta]= j^1\{\pr \xi, \pr \eta\}+ \Lie_{\xi}(D\eta)- \Lie_{\eta}(D\xi).\]
\end{remark}

In particular, we obtain the following integrability result, which should be compared with the one of \cite{Jacobi} (and please compare the proofs as well!).

\begin{corol} Given a Jacobi structure in the wide sense $(L, \{\cdot, \cdot\})$ over $M$, if the associated Lie algebroid $\Jet^1L$ comes from an
$s$-simply connected Lie groupoid $\Sigma$, then $\Sigma$ carries a contact hyperfield $\H$ making it into a contact groupoid in the wide sense;
$\H$ is uniquely characterized by the fact that the associated Spencer operator coincides with the classical one.
If $(L, \{\cdot, \cdot\})$ comes from a Jacobi structure $(\Lambda, R)$, then we end up with a contact groupoid $(\Sigma, \theta, r)$. 
\end{corol}

\section{Proof of Theorem \ref{t1}}
\label{sec: linearization}

\subsection{Rough idea and some heuristics behind the proof}
\label{Rough idea and some heuristics behind the proof}

Let us briefly indicate the intuition behind our approach (the pseudogroup point of view). The main idea is to reinterpret $\theta$ in terms of bisections: it gives rise to (and it is determined by) a family of $k$-forms $\{ \theta_b: b\in \textrm{Bis}(\G)\}$; here $\theta_b\in \Omega^k(M, E)$ is obtained by pull-backing $\theta$ to $M$ via $b$ and using the action of $\G$ on $E$. In other words, $\theta$ is encoded in the map
\begin{equation} \label{Theta}
\Theta: \textrm{Bis}(\G)\rmap \Omega^k(M, E),\ \ b\mapsto \theta_b ;
\end{equation}
the multiplicativity of $\theta$ translates into a cocycle condition for $\Theta$ on the group $\textrm{Bis}(\G)$. Hence, morally (because we are in infinite dimensions), the infinitesimal counterpart of $\theta$ is encoded in the linearization $\ve(\Theta)$ of $\Theta$ (as in Subsection \ref{1-cocyles and relations to the van Est map}). While $\Gamma(A)$ plays the role of the Lie algebra of $\textrm{Bis}(\G)$ (cf. Remark \ref{right-invariance}), one arrives at $\ve(\Theta)= D$ given in the theorem. However, to prove the theorem, we have to avoid the infinite dimensional problem and work (still in the spirit of Lie pseudogroups) with jet spaces: since $\Theta$ depends only on first order jets of bisections, it can be reinterpreted as a finite dimensional object -- a map
\[ c: \Jet^1\G  \longrightarrow \hom(\wedge^kTM, E),\]
which is a $1$-cocycle for the groupoid $\Jet^1\G$. Hence, instead of applying the integration of cocycles (as in Proposition \ref{prop: Van Est}) to the infinite dimensional $\textrm{Bis}(\G)$, we will apply it to the groupoid $\Jet^1\G$.
Here one encounters a small technical problem: $\Jet^1\G$ may have $s$-fibers which are not simply connected (not even connected), so we will have to pass to the closely related groupoid $\widetilde{\Jet^1\G}$, which has the same Lie algebroid $\Jet^1A$, but which is $s$-simply connected:
\[\tilde{c}: \widetilde{\Jet^1\G} \longrightarrow \hom(\wedge^kTM, E).\]
Then one has to concentrate on the linearization cocycle 
\[ \eta: \Jet^1A \longrightarrow \hom(\wedge^kTM, E),\]
which, together with the decomposition (\ref{J-decomposition}) in mind, is precisely the pair $(D, l)$ consisting of the Spencer operator and its symbol. 
The rest is about working out the details and finding out the precise equations that $c$ and $\eta$ have to satisfy. 

Throughout this section, $\G$ denotes a Lie groupoid over $M$, $E$ is a representation of $\G$, $\Jet^1\G$ denotes the Lie groupoid of $1$-jets of bisections of $\G$. Each one of the next subsections is devoted
to one of the 1-1 correspondences
\[ \theta \longleftrightarrow c \longleftrightarrow \tilde{c} \longleftrightarrow \eta .\]

\subsection{From multiplicative forms to differentiable cocycles}\label{sec: step 1}

Recall that $\lambda: \Jet^1\G \to \GL(TM)$ denotes the canonical representation of $\Jet^1\G$ on $TM$ and $\Ad: \Jet^1\G \to \GL(A)$ denotes the adjoint representation of $\Jet^1\G$ on the Lie algebroid $A$ of $\G$. Combining $\lambda$ with the action of $\G$ on $E$, 
\[ \hom(\wedge^kTM, E)\]
becomes a representation of $\Jet^1\G$. Using the notations from  Subsection \ref{1-cocyles and relations to the van Est map}, we will denoted the associated space of $1$-cocycles by
\[ Z^{1}(\Jet^1\G, \hom(\wedge^kTM, E)).\]
Moreover, as in Remark \ref{when working with jets}, we will view the elements of $\Jet^1\G$ as splittings $\sigma_g: T_{s(g)}\to T_g\G$ of $\d s$. In particular, for $\sigma_g, \sigma_g'\in \Jet^1\G$ sitting above the same $g\in \G$, $\sigma_g- \sigma_g'$ takes values in $Ker(ds)_g$; identifying the last space with $A_{t(g)}$, we consider the resulting map
\[ \sigma_g \ominus \sigma_g':= R_{g^{-1}}\circ (\sigma_g - \sigma'_g): T_{s(g)}M\To A_{t(g)}.\]\\

\begin{prop}\label{prop: multiplicative forms as cocycles}
There is a one-to-one correspondence between multiplicative $k$-forms $\theta \in \Omega^k(\G, t^{\ast}E)$, and pairs $(c, l)$ with 
\[\left\{\begin{aligned}
c\in Z^{1}(\Jet^1\G, \hom(\wedge^kTM, E))\\
l: A \longrightarrow \hom(\wedge^{k-1}TM, E) 
\end{aligned}\right.\]
satisfying the following equations:
\begin{equation}\label{eq: condition 1}\begin{aligned}
c&(\sigma_g)(\lambda_{\sigma_g}v_1, \ldots , \lambda_{\sigma_g}v_k) - c(\sigma_g')(\lambda_{\sigma_g'}v_1, \ldots , \lambda_{\sigma_g'}v_k) = \\
&=\sum_{i = 1}^k(-1)^{i+1}l((\sigma_g \ominus \sigma_g')(v_i))(\lambda_{\sigma_g}v_1,\ldots,\lambda_{\sigma_g}v_{i-1}, \lambda_{\sigma_g'}v_{i+1},\ldots, \lambda_{\sigma_g'}v_k), 
\end{aligned}\end{equation}
\begin{equation}\label{eq: condition 2}
i_{\rho(\al)}l(\be) = -i_{\rho(\be)}l(\al),
\end{equation}
\begin{equation}\label{eq: condition 3}\begin{aligned}
l(\Ad_{\sigma_g}\al)&(\lambda_{\sigma_g}v_1, \ldots , \lambda_{\sigma_g}v_{k-1}) - g\cdot l(\al)(v_1,\ldots,v_{k-1}) = \\
&=c(\sigma_g)(\lambda_{\sigma_g}\rho(\al),\lambda_{\sigma_g}v_1, \ldots , \lambda_{\sigma_g}v_{k-1}), 
\end{aligned}\end{equation}
for all splittings $\sigma_g, \sigma_g'\in \Jet^1\G$, $v_1, \ldots, v_k \in T_{x}M$, and $\al, \be \in A_x$, where $x= s(g)$. 
 \end{prop}
 
 
 For the proof of the proposition we will need the following lemma:
 \begin{lemma}\label{lemma: theta vertical}\rm
For any multiplicative form $\theta \in \Omega^k(\G,t^{\ast}E)$, and any $\al_g \in \ker (\d s)_g$, we have that
\begin{equation}\label{eq: theta vertical 1}
\theta_g(\al_g, X_1, \ldots, X_{k-1}) = \theta_{t(g)}(\al_{t(g)},(\d t)_g(X_1), \ldots, (\d t)_g(X_{k-1})),
\end{equation}
for all $X_1, \ldots, X_{k-1} \in T_g\G$, where $\al_{t(g)} = R_{g^{-1}}(\al_g)$.
 \end{lemma}
 
 \begin{proof}
Notice that we can express $\al_g$ and $X_i$ as 
\[\al_g = (\d R_g)_{t(g)}(\al_{t(g)}) = (\d m)_{(t(g),g)}(\al_{t(g)}, 0_g), \  X_i = (\d m)_{(t(g),g)}((\d t)_g(X_i), X_i).\]
Equation \eqref{eq: theta vertical 1} then follows from the multiplicativity equation (\ref{eq: multiplicative}) on the vectors (tangent to $\G_2$) $(\al_{t(g)}, 0_g)$, $((\d t)_g(X_1), X_1), \ldots$. 
 \end{proof}\\

\begin{proof}[Proof of proposition \ref{prop: multiplicative forms as cocycles}]
In one direction, given $\theta \in \Omega^k(\G, t^{\ast}E)$, 
\[c(\sigma_g)(w_1, \ldots, w_k) = \theta_g(\sigma_g(\lambda_{\sigma_g}^{-1}(w_1)), \ldots, \sigma_g(\lambda_{\sigma_g}^{-1}(w_k))), \]
 for all $\sigma_g\in \Jet^1\G$ and $w_1, \ldots, w_k \in T_{t(g)}M$, and 
\[l(\al)(v_1, \ldots, v_{k-1}) = \theta_x(\al, v_1, \ldots, v_{k-1}),\]
for all $\al \in A_x$, and $v_1, \ldots, v_{k-1}\in T_xM$. The desired equations for $(c,l)$ will be proven for $k= 2$, which reveals all the necessary
arguments but keeps the notations simpler. Note that in this case Lemma \ref{lemma: theta vertical} translates into
\begin{equation}\label{eq: theta vertical}
\theta_g(\al_g, X) = l(\al_{t(g)})((\d t)_g(X)),
\end{equation}
for all $X \in T_g\G$, $\al_g \in \ker (\d s)_g$, where $\al_{t(g)} = R_{g^{-1}}(\al_g)$. To prove \eqref{eq: condition 1} (for $k= 2$), using the definition of $c$, we find that the left hand side of the
equation is 
\[\begin{aligned}
\theta_g(\sigma_g(v_1), & \sigma_g(v_2)) - \theta_g(\sigma_g'(v_1),\sigma_g'(v_2)) = \\
 & =\theta_g(\sigma_g(v_1) - \sigma'_g(v_1),  \sigma_g(v_2)) + \theta_g(\sigma_g'(v_1), \sigma_g(v_2) - \sigma_g'(v_2)).
\end{aligned}\]
Applying \eqref{eq: theta vertical} with $\al_{g}= \sigma_g(v_i) - \sigma_g'(v_i)$, $i\in \{1, 2\}$, we obtain \eqref{eq: condition 1}. The same \eqref{eq: theta vertical}, combined
with the skew symmetry of $\theta$ gives \eqref{eq: condition 2}:
\[l(\al)(\rho(\be)) = \theta(\al,\be) = -\theta(\be, \al) = -l(\be)(\rho(\al)).\]

Next we prove \eqref{eq: condition 3}. Using  the formula \eqref{eq: Ad} for the adjoint representation,
\[
l(\Ad_{\sigma_g}\al)(\lambda_{\sigma_g}v) = 
\theta_{t(g)}(R_{g^-1} \circ (\d m)_{(g, s(g))}(\sigma_g(\rho(\al)), \al),\lambda_{\sigma_g}v),\]
Using again the equation \eqref{eq: theta vertical}, the last expression is 
\[\theta_g((\d m)_{(g, s(g))}(\sigma_g(\rho(\al)), \al),\sigma_g(v)).\] 
Combining with $\sigma_g(v) = (\d m)_{(g,s(g))}(\sigma_g(v), v)$ and then applying the multiplicativity equation for $\theta$, we arrive at the
right hand side of \eqref{eq: condition 3}.

We are left with proving the cocycle equation $\delta c = 0$. Let $(\sigma_g, \sigma_h) \in (\Jet^1\G)_2$ be a pair of composable arrows. Then $\delta c(\sigma_g, \sigma_h)$ is the map $\wedge^2T_{t(g)}M \to E_{t(g)}$,
\[ \delta c(\sigma_g, \sigma_h)(w_1, w_2) = c(\sigma_g)(w_1, w_2)  +  g\cdot(c(\sigma_h)(\lambda_{\sigma_g}^{-1}w_1,  \lambda_{\sigma_g}^{-1}w_2)) -  c(\sigma_g\cdot\sigma_h)(w_1, w_2).\]
Let $v_i = \lambda_{\sigma_g}^{-1}w_i \in T_{t(h)}M$. For the sum of the first two terms in the right hand side, after applying the definition of $c$, we find
\[\theta_g(\sigma_g(v_1), \sigma_g(v_2))+  g\cdot\theta_h(\sigma_h(\lambda_{\sigma_h}^{-1}v_1),  \sigma_h(\lambda_{\sigma_h}^{-1}v_2)).\]
For the last term, using the description (\ref{mult-i-J1}) for $\sigma_g\cdot\sigma_h$, we find 
\[ \theta_{gh}((\d m)_{(g,h)}(\sigma_g(v_1),\sigma_h(\lambda_{\sigma_h}^{-1}v_1)),  (\d m)_{(g,h)}(\sigma_g(v_2),\sigma_h(\lambda_{\sigma_h}^{-1}v_2)).\]
Finally, the multiplicativity of $\theta$ implies that the last two expressions coincide. 


For the reverse direction, let $c$ and $l$ be given, and we construct $\theta$. In order to avoid clumsier notations, we extend $l$ to the entire
$\ker \d s$: for $\alpha_g\in \ker(\d s)_g$:
\[ l(\alpha_g):= l(R_{g^{-1}}\alpha_g).\]
Given $g$, choose $\sigma_g \in \Jet^1\G$ and use it to split a vector $X \in T_g\G$ into
\[X = \sigma_g(v) + \al_X,\]
where $v = (\d s)_g(X) \in T_{s(g)}M$, and $\al_X = X - \sigma_g(v) \in \ker(\d s)_g$. 
We define
\[\begin{aligned}
&\theta_g(X_1, \ldots, X_k) = c(\sigma_g)(\lambda_{\sigma_g}v_1, \ldots \lambda_{\sigma_g}v_k) + \\
&+ \sum_{p+q=k}\sum_{\tau \in S(p,q)}(-1)^{|\tau|}l(\al_{X_{\tau(1)}})(\rho(\al_{X_{\tau(2)}}),\ldots, \rho(\al_{X_{\tau(p)}}), \lambda_{\sigma_g}v_{\tau(p+1)}, \ldots, \lambda_{\sigma_g}v_{\tau(k)}),
\end{aligned}\]
where $p \geq 1$, and the second summation is taken over all $(p,q)$-shuffles of $\set{1,\ldots,k}$. The rest of the proof will be given again in the case $k= 2$, when the previous formula becomes
\[\begin{aligned}
\theta_g(X_1, X_2) = c(\sigma_g)&(\lambda_{\sigma_g}v_1,\lambda_{\sigma_g}v_2)  - l(\al_{X_2})(\lambda_{\sigma_g}v_1) +\\
&+ l(\al_{X_1})(\lambda_{\sigma_g}v_2) + l(\al_{X_1})(\rho(\al_{X_2})).
\end{aligned}\]
Note that \eqref{eq: condition 2} implies that $\theta_g$ is skew-symmetric. Let us show that it does not depend on the choice of $\sigma_g$.  
Choose another splitting $\sigma_g'$ and write $\al'_X = X-\sigma_g'(v)$. Let $\theta'_g$ be the form obtained by using the splitting $\sigma_g'$. Then
\[\begin{aligned}
(\theta_g - \theta_g')(X_1, X_2) =  c(&\sigma_g)(\lambda_{\sigma_g}v_1, \lambda_{\sigma_g}v_2) - c(\sigma_g')(\lambda_{\sigma_g'}v_1, \lambda_{\sigma_g'}v_2) +\\
& - l(\al_{X_2})(\lambda_{\sigma_g}v_1) + l(\al'_{X_2})(\lambda_{\sigma_g'}v_1)\\
&+ l(\al_{X_1})(\lambda_{\sigma_g}v_2) - l(\al'_{X_1})(\lambda_{\sigma'_g}v_2) + \\
& + l(\al_{X_1})(\rho(\al_{X_2})) - l(\al'_{X_1})(\rho(\al'_{X_2})).
\end{aligned}\]

Let us denote by $\al_{v_i} = \sigma_g(v_i) - \sigma'_g(v_i)$ and notice that 
\begin{equation}\label{eq: alpha - alpha'}
\al_{X_i} - \al'_{X_i} = - \al_{v_i}, \text{ and } \lambda_{\sigma_g}v_i - \lambda_{\sigma'_g}v_i = \rho(\al_v).
\end{equation} 
By using \eqref{eq: condition 1} and the polarization formula for $\theta$, it follows that
\[\begin{aligned}
(\theta_g - \theta_g')(X_1, X_2) &=  l(\al_{v_1})(\lambda_{\sigma_g}v_2) - l(\al_{v_2})(\lambda_{\sigma_g'}v_1) +\\
&\quad + l(\al_{v_2})(\lambda_{\sigma_g}v_1) - l(\al'_{v_2})(\rho(\al_{v_1}))\\
&\quad- l(\al_{v_1})(\lambda_{\sigma_g}v_2) + l(\al'_{X_1})(\rho(\al_{v_2})) + \\
&\quad - l(\al_{v_1})(\rho(\al_{X_2})) - l(\al'_{X_1})(\rho(\al_{v_2})).
\end{aligned}\]

Thus, if we substitute into this expression the consequence
\[l(\al_{v_2})(\lambda_{\sigma_g'}v_1) = l(\al_{v_2})(\lambda_{\sigma_g}v_1) - l(\al_{v_2})(\rho(\al_{v_1}))\]
of \eqref{eq: alpha - alpha'}, almost all of the terms cancel out two-by-two and we are left with
\[\begin{aligned}
(\theta_g - \theta_g')(X_1, X_2) &= l(\al_{v_2})(\rho(\al_{v_1})) - l(\al'_{X_2})(\rho(\al_{v_1})) - l(\al_{v_1})\rho(\al_{X_2})\\
&= l(\al_{v_2} - \al'_{X_2})(\rho(\al_{v_1})) - l(\al_{v_1})(\rho(\al_{X_2}))\\
&= -l(\al_{X_2})(\rho(\al_{v_1})) - l(\al_{v_1})(\rho(\al_{X_2})).
\end{aligned}\]
Because of \eqref{eq: condition 2}, this expression vanishes and $\theta$ is well defined. 

We are left with the verification that $\theta$ is multiplicative. Let $g,h \in \G$ be composable arrows, and let $(X_i,Y_i) \in T_{(g,h)}\G_2$ so that $(\d t)_h(Y_i) = (\d s)_g(X_i)$. We fix splittings $\sigma_g,\sigma_h \in \Jet^1\G$ and use them to write
\[X_i = \al_i + \sigma_g(v_i), \text{ and } Y_i = \be_i +\sigma_h(w_i).\]
From $(X_i,Y_i) \in T_{(g,h)}\G_2$ it follows that $v_i = \rho(\be_i) + \lambda_{\sigma_h}(w_i)$, hence
\[X_i = \al_i +\sigma_g(\rho(\beta_i))+ \sigma_g(\lambda_{\sigma_h}w_i).\]
Decomposing
\[\d m(X_i,Y_i) = \d m(\al_i,0) + \d m(\sigma_g(\rho(\beta_i)), \be_i)+ \d m(\sigma_g(\lambda_{\sigma_h}w_i) ,\sigma_h(w_i)).\]
$m^{\ast}\theta_{(g,h)}((X_1,Y_1),  (X_2, Y_2))$ gives six types of terms (where $1\leq i \neq j \leq2$):
\begin{description}
\item[Type 1:]  $\theta_{gh}(\d m(\al_1,0), \d m(\al_2, 0)),$
\item[Type 2:]  $\theta_{gh}(\d m(\al_i,0),\d m(\sigma_g(\rho(\beta_j)), \be_j)),$
\item[Type 3:]   $\theta_{gh}(\d m(\al_i,0),  \d m(\sigma_g(\lambda_{\sigma_h}w_j) ,\sigma_h(w_j))),$
\item[Type 4:]  $\theta_{gh}(\d m(\sigma_g(\rho(\beta_1)), \be_1),\d m(\sigma_g(\rho(\beta_2)), \be_2)),$
\item[Type 5:]    $\theta_{gh}(\d m(\sigma_g(\rho(\beta_i)), \be_i),  \d m(\sigma_g(\lambda_{\sigma_h}w_j) ,\sigma_h(w_j)))$
\item[Type 6:]    $\theta_{gh}(\d m(\sigma_g(\lambda_{\sigma_h}w_1) ,\sigma_h(w_1)),  \d m(\sigma_g(\lambda_{\sigma_h}w_2) ,\sigma_h(w_2)))$
\end{description}

In order to simplify the terms of type 1,2, and 3, we note that $(\d m)_{(g,h)}(\al_g,0_h) = R_h(\al_g)$ for all $\alpha_g\in \ker(ds)_g$ and thus, by the definition of $\theta$
\[\theta_{gh}(\d m(\al_1,0), \d m(\al_2, 0)) = l(\al_1)(\rho(\al_2)),\]\
\[\theta_{gh}(\d m(\al_i,0),\d m(\sigma_g(\rho(\beta_j)), \be_j)) = l(\al_i)(\lambda_{\sigma_g}\rho(\be_j)), \]\
\[\theta_{gh}(\d m(\al_i,0),  \d m(\sigma_g(\lambda_{\sigma_h}w_j) ,\sigma_h(w_j))) = l(\al_i)(\lambda_{\sigma_g\sigma_h}w_j).\]\

On the other hand, using again the formula \eqref{eq: Ad} for the adjoint action, as well as condition \eqref{eq: condition 3}, we simplify the terms of type 4 and 5 into
\[\begin{aligned}
\theta_{gh}(\d m(\sigma_g(\rho(\beta_1)), \be_1),&\d m(\sigma_g(\rho(\beta_2)), \be_2)) = l(\Ad_{\sigma_g}\be_1)(\lambda_{\sigma_g}(\rho(\be_2))) \\
&=c(\sigma_g)(\lambda_{\sigma_g}\rho(\be_1), \lambda_{\sigma_g}\rho(\be_2)) +g\cdot l(\be_1)(\rho(\be_2))
\end{aligned}\]\
\[\begin{aligned}
\theta_{gh}(\d m(\sigma_g(\rho(\beta_i)), \be_i), & \d m(\sigma_g(\lambda_{\sigma_h}w_j) ,\sigma_h(w_j))) = l(\Ad_{\sigma_g}\be_i)(\lambda_{\sigma_g\sigma_h}w_j) \\
&= c(\sigma_g)(\lambda_{\sigma_g}\rho(\be_i),\lambda_{\sigma_g\sigma_h}w_j) + g\cdot l(\be_i)(\lambda_{\sigma_h}w_j).
\end{aligned}\]

Finally, we use condition \eqref{eq: condition 2} to express
\[\begin{aligned}
\theta_{gh}(\d m(\sigma_g(\lambda_{\sigma_h}w_1) ,\sigma_h(w_1)), & \d m(\sigma_g(\lambda_{\sigma_h}w_2) ,\sigma_h(w_2))) = \\&=\theta_g(\sigma_g(\lambda_{\sigma_h}w_1),\sigma_g(\lambda_{\sigma_h}w_2)) + g\cdot \theta_h(\sigma_h(w_1),\sigma_h(w_2)).
\end{aligned}\]
Adding everything up, we recognize $\theta_g(X_1,X_2) +g\cdot \theta_h(Y_1,Y_2)$, 
thus concluding the proof of the proposition.
\end{proof}


\subsection{Realizing source-simply connectedness}\label{sec: step 2}

As mentioned at the beginning of the section, we need to pass from $\Jet^1\G$ to $\widetilde{\Jet^1\G}$, the source simply connected Lie groupoid with the same Lie algebroid $\Jet^1A$ as $\Jet^1\G$. 
For an explicit construction of $\widetilde{\Jet^1\G}$, one puts together the universal covers of the $s$-fibers of $\Jet^1\G$ with base points the units (see e.g. \cite{CrainicFernandes} for the general discussion). It comes equipped with a groupoid map
\[p: \widetilde{\Jet^1\G} \To \Jet^1\G,\]
whose image is the subgroupoid $(\Jet^1\G)^{0}$ made of the connected component of the identities in $\Jet^1\G$. For elements in $\widetilde{\Jet^1\G}$ we will use the notation 
$\sigma_g$ whenever we want to indicate the point $g\in \G$ onto which $\sigma_g$ projects. For $X\in T_{s(g)}\G$ we will use the notation
\[ \sigma_g(X):= p(\sigma_g)(X)\]
and, for $\sigma_g, \sigma_g'\in \widetilde{\Jet^1\G}$, consider
\[ \sigma_{g}\ominus \sigma_g':= p(\sigma_g)\ominus p(\sigma_g'): T_{s(g)}M\To A_{t(g)}.\]
We will use the map $p$ to pull-back structures from $\Jet^1\G$ to $\widetilde{\Jet^1\G}$. 
For instance, any representation of $\Jet^1\G$ can also be seen as a representation of $\widetilde{\Jet^1\G}$
and there is an induced pull-back map at the level of the resulting cocycles. In particular, we will consider 
\begin{equation}
\label{p-ast}
p^{\ast}: Z^1(\Jet^1\G,\hom(\wedge^kTM,E)) \to Z^{1}(\widetilde{\Jet^1\G}, \hom(\wedge^kTM,E)).
\end{equation}
It is clear that the pairs $(c, l)$ of Proposition \ref{prop: multiplicative forms as cocycles}, and the equations that they satisfy, have an analogue to with $\Jet^1\G$ replaced by
$\widetilde{\Jet^1\G}$, giving rise to pairs $(\tilde{c}, l)$ satisfying identical equations.

\begin{prop}\label{corol: passage to cover}
Let $\G$ be an $s$-simply connected Lie groupoid. Then $(c, l) \mapsto (p^{\ast}(c), l)$ defines a 
1-1 correspondence between pairs $(c, l)$ satisfying the conditions from Proposition \ref{prop: multiplicative forms as cocycles}
and pairs $(\tilde{c}, l)$ consisting of 
\[\left\{\begin{aligned}
\tilde{c}\in Z^{1}(\widetilde{\Jet^1\G}, \hom(\wedge^kTM, E))\\
l: A \longrightarrow \hom(\wedge^{k-1}TM, E) 
\end{aligned}\right.\]
satisfying the conditions from Proposition \ref{prop: multiplicative forms as cocycles} but with $\Jet^1\G$ replaced by $\widetilde{\Jet^1\G}$. 
\end{prop}

\begin{proof}
Of course, the statement is about $c\mapsto p^{\ast}c$, i.e. we can fix $l$. We begin by showing that $p^{\ast}$ is injective when restricted to the set of $c$ for which \eqref{eq: condition 1} holds. To do so we first show that $c$  is determined by its value on the the Lie groupoid $(\Jet^1\G)^0$ whose $s$-fibers are the connected component of the identity in the $s$-fibers of $\Jet^1\G$. Observe that for any $g \in \G$, there exists $\sigma_g \in (\Jet^1\G)^0$. In fact, since $(\Jet^1\G)^0 \to \G$ is a submersion, and $\G$ is $s$-connected, we can lift any path in $s^{-1}(s(g))$, starting at the identity and ending at $g$, to a path in $(\Jet^1\G)^0$ starting at the identity and ending over $g$. The corresponding end point is an element $\sigma_g \in (\Jet^1\G)^0$. It follows from \eqref{eq: condition 1} that for any other $\sigma'_g \in \Jet^1\G$, $c(\sigma'_g)$ is determined by $c(\sigma_g)$, and $l$.  

Next, we note that if $p^{\ast}c = p^{\ast}c'$, then $c$ and $c'$ coincide on $(\Jet^1\G)^0$. In fact, for any $\sigma_g \in (\Jet^1\G)^0$ we can find a path inside the $s$-fiber of $\sigma_g$, joining the identity $(\d u)_{s(g)}$ of $(\Jet^1\G)^0$, and $\sigma_g$. Taking its homotopy class, this path gives rise to an element $\xi_g$ of $\widetilde{\Jet^1\G}$ which projects to $\sigma_g$. But then,
\[c(\sigma_g) = c(p(\xi_g)) = c'(p(\xi_g)) = c'(\sigma_g).\]

Finally, we prove that if $(\tilde{c},l)$ satisfies \eqref{eq: condition 1}, then $\tilde{c}$ lies in the image of $p^{\ast}$. For this, note that if $p(\xi_g) = p(\xi'_g)$, then the actions of $\xi_g$ and $\xi'_g$ on $TM$ coincide. Moreover, they induce the same splittings of $(\d s)_g$. Thus, the right hand side of \eqref{eq: condition 1} vanishes, which implies that $\tilde{c}(\xi_g) = \tilde{c}(\xi'_g)$. It follows that $\tilde{c}$ induces a map $c: \Jet^1\G \To t^{\ast}\hom(\wedge^kTM,E)$ such that $p^{\ast}c = \tilde{c}$.

Thus, we have just proven that $p^{\ast}$ determines a one-to-one correspondence between $c \in Z^1(\Jet^1\G,\hom(\wedge^kTM,E))$ such that $(c,l)$ satisfies \eqref{eq: condition 1}, and $\tilde{c} \in Z^{1}(\widetilde{\Jet^1\G}, \hom(\wedge^kTM,E))$ such that $(\tilde{c},l)$ satisfy \eqref{eq: condition 1}. A simple verification shows that $(c,l)$ satisfies \eqref{eq: condition 2} and \eqref{eq: condition 3} if and only if $(\tilde{c},l)$ satisfies \eqref{eq: condition 2} and \eqref{eq: condition 3}. This concludes the proof.
\end{proof}

\subsection{Passing to algebroid cocycles}\label{sec: algebroid cocycles}


\begin{prop}\label{prop: algebroid cocycles}
Let $\G$ be $s$-simply connected. Then there is a one-to-one correspondence between pairs $(\tilde{c}, l)$ as in Proposition \ref{corol: passage to cover} 
and pairs $(\eta,l)$ with
\[\left\{\begin{aligned}
\eta \in Z^1(\Jet^1A, \hom(\wedge^kTM, E))\\
l: A \longrightarrow \hom(\wedge^{k-1}TM, E) 
\end{aligned}\right.\]
 satisfying the equations:
\begin{equation}\label{eq: condition 1''}
\eta(\d f \otimes \al) = \d f \wedge l(\al),
\end{equation}
\begin{equation}\label{eq: condition 2''}
i_{\rho(\al)}l(\be) = - i_{\rho(\be)}l(\al),
\end{equation}
\begin{equation}\label{eq: condition 3''}
l([\al,\be]) - \Lie_{\al}l(\be) = i_{\rho(\al)}\eta(\jet^1\be),
\end{equation}
for all $\al, \be \in \Gamma(A)$, and all $f \in \mathrm{C}^{\infty}(M)$.
\end{prop}

For \eqref{eq: condition 1''} we are using the inclusion $i$ from the exact sequence
\[ 0\to \textrm{Hom}(TM, A)\stackrel{i}{\to} \Jet^1A \stackrel{\pr}{\to} A\to 0\]

\begin{proof}
We use the isomorphism 
\[\ve: Z^1(\widetilde{\Jet^1\G},\hom(\wedge^kTM, E)) \To Z^1(\Jet^1A, \hom(\wedge^kTM, E))\]
induced by the van Est map (Proposition \ref{prop: Van Est}); of course, $\eta= \ve(\tilde{c})$. 
Fix $(\tilde{c},l)$ and $x\in M$. We prove that \eqref{eq: condition 1} and \eqref{eq: condition 3} for $(\tilde{c},l)$ are equivalent to \eqref{eq: condition 1''} and \eqref{eq: condition 3''} for $(\ve(\tilde{c}),l)$.

We start with the equivalence of \eqref{eq: condition 3} with \eqref{eq: condition 3''}. We interpret $l$ as 
\[ l\in \Gamma(A^{\ast}\otimes\wedge^{k-1}T^{\ast}M\otimes E)= C^0(\widetilde{\Jet^1\G}, A^{\ast}\otimes\wedge^{k-1}T^{\ast}M\otimes E),\]
a zero-cochain on $\widetilde{\Jet^1\G}$. Of course, $\ve(l)= l$, with $l$ interpreted as a $0$-cochain on $\Jet^1A$. 
On the other hand, the anchor $\rho$ induces a morphism of representations
\[\rho^{\ast}: \wedge^kT^{\ast}M\otimes E \to A^{\ast}\otimes\wedge^{k-1}T^{\ast}M\otimes E\] 
hence also a map of complexes
\[\rho^{\ast}: C(\widetilde{\Jet^1\G}, \wedge^kT^{\ast}M\otimes E) \to C(\Jet^1A, A^{\ast}\otimes\wedge^{k-1}T^{\ast}M\otimes E)\]
and similarly on the algebroid cohomology, compatible with $\ve$. In particular,
\[ \ve(d(l)- \rho^{\ast}(\tilde{c}))= d(\ve(l))- \rho^{\ast}(\ve(\tilde{c}))= d(l)- \rho^{\ast}(\eta).\]
Finally, note that \eqref{eq: condition 3} is just the explicit form of the equation $d(l)= \rho^{\ast}(\tilde{c})$, while \eqref{eq: condition 3''} is just $d(l)= \rho^{\ast}(\eta)$; 
hence one just uses the injectivity of $\ve$.

We are left with proving the equivalence of \eqref{eq: condition 1} with \eqref{eq: condition 1''}. We fix $x\in M$ and we show that \eqref{eq: condition 1''} is satisfied at $x$ if and only
if \eqref{eq: condition 1} is satisfied for all $g$ that start at $x$. In the sequence
\[ \widetilde{\Jet^1(\G)}\stackrel{p}{\To} \Jet^1(\G) \stackrel{\pr}{\To} \G,\]
we consider the $s$-fibers of the three groupoids above $x$, denoted
\[ \tilde{P}\To P\To B .\]
Both $P$ and $\tilde{P}$ becomes principal bundles over $B$, with structure groups 
\[ K= \pr^{-1}(1_x),\ \hat{K}= (\pr\circ p)^{-1}(1_x),\]
respectively (the action is the one induced by the groupoid multiplication). Note also that
the map $p: \tilde{P}\to P$ has as image the connected component $P^0$ of $P$ containing the unit at $x$
and $p: \tilde{P}\to P^0$ is a covering projection. 
Since $B$ is assumed to be simply connected, the following is immediate. 

\begin{lemma}\label{lemma: K}
$\hat{K}$ is connected. 
\end{lemma}

Assume now that $(\ve(\tilde{c}),l)$ satisfies \eqref{eq: condition 1} for all $g\in B$. Let $\eta = \ve(\tilde{c})$ and $\d f \otimes \al \in T^{\ast}M\otimes A$. Using $p : \tilde{P} \to P$ to identify a neighborhood of the identity in $\tilde{P}$, with a neighborhood of the identity in $P$, we can view, for $\eps$ small enough,
\[\gamma_x(\eps) = (\d u)_x + \eps((\d f)_x \otimes \al_x)\]
as a path in $\tilde{P}$ such that 
\[\gamma(0) = (\d u)_x, \quad \frac{\d}{\d\eps}\big{|}_{\eps = 0} \gamma_x = (\d f)_x \otimes \al_x.\] 

Since for each $\eps$, $\gamma_x(\eps)$ acts trivially on $E$, it follows that 
\begin{equation}\label{eq: around the identity}
\eta((\d f)_x \otimes \al_x) = \frac{\d}{\d\eps}\big{|}_{\eps = 0} \tilde{c}(\gamma_x(\eps))
\end{equation}
However, since $\tilde{c}$ is a cocycle, it follows that $\tilde{c}(\d_xu) = 0$, thus \eqref{eq: condition 1} implies that
\[\tilde{c}(\gamma_x(\eps))(v_1, \ldots, v_k) = \eps((\d f)_x\wedge l(\al_x))(v_1,\ldots, v_k), \]
for all $\eps$ and all $v_i \in T_xM$. By differentiating at $\eps = 0$ one obtains \eqref{eq: condition 1''} at $x$.

We now prove the converse, namely that \eqref{eq: condition 1''} at $x$ implies \eqref{eq: condition 1} at all $g\in B$. Fix $g$ and let $\xi_g$ and $\xi_g'$ be two elements of $\widetilde{\Jet^1\G}$ which lie over $g$. Let $\gamma_1$ be any path in $\tilde{P}$ joining $(\d u)_x$ to $\xi_g$, and let $\gamma_2$ be any path in the fiber of $\tilde{P} \to B$ joining $\xi_g$ to $\xi'_g$, which exists because of Lemma \ref{lemma: K}. Furthermore, we may assume that
\[\gamma_1(\eps) = \xi_g \text{ for all } \frac{1}{2}\leq\eps\leq 1,\ \ \gamma_2(\eps) = \xi_g \text{ for all } 0\leq\eps\leq\frac{1}{2}.\]\\
Thus, we obtain two smooth paths
\[\gamma_{\xi_g}(\eps) = \left\{\begin{aligned}&\gamma_1(2\eps) &\text{ if } 0 \leq \eps \leq\frac{1}{2}\\ 
&\xi_g &\text{ if } \frac{1}{2}\leq \eps\leq 1,
\end{aligned}\right. ,\ \ \gamma_{\xi_g'}(\eps) = \left\{\begin{aligned}&\gamma_1(2\eps) &\text{ if } 0 \leq \eps \leq\frac{1}{2}\\ 
&\gamma_2(2\eps - 1) &\text{ if } \frac{1}{2}\leq \eps\leq 1
\end{aligned}\right.\] 
\vskip0.5ex
Finally, we consider the path $f: [0,1] \to E$ given by
\[\begin{aligned}
&f(\eps) = \tilde{c}(\gamma_{\xi_g}(\eps))(\lambda_{\gamma_{\xi_g}(\eps)}v_1, \ldots , \lambda_{\gamma_{\xi_g}(\eps)}v_k) - \tilde{c}(\gamma_{\xi_g'}(\eps))(\lambda_{\gamma_{\xi_g'}(\eps)}v_1, \ldots , \lambda_{\gamma_{\xi_g'}(\eps)}v_k) + \\
&-\sum_{i = 1}^k(-1)^{i+1}l((\gamma_{\xi_g} \ominus \gamma_{\xi_g'})(\eps)(v_i))(\lambda_{\gamma_{\xi_g}(\eps)}v_1,\ldots,\lambda_{\gamma_{\xi_g}(\eps)}v_{i-1}, \lambda_{\gamma_{\xi_g'}(\eps)}v_{i+1},\ldots, ) 
,\end{aligned} \]
where $v_i \in T_{s(g)}M$. We must show that $f(\eps)$ is contained in the zero section $E$ for all $\eps \in [0,1]$. It is clear that this is true for $\eps \in [0, 1/2]$. On the other hand, for $\eps \in [1/2, 1]$, the path $f(\eps)$ lies inside the fiber $E_{t(g)}$. Thus, it suffices to show that the derivative of $f$ at $\eps$ vanishes for all $\eps \in [1/2,1]$. However, by Proposition \ref{prop: appendix} this is reduced to the computation \eqref{eq: around the identity} performed at the identity $(\d u)_{s(g)}$, which vanishes by virtue of \eqref{eq: condition 1''}.
\end{proof}\\

\begin{proof}[End of proof of Theorem \ref{t1}] 
We put together Propositions \ref{prop: multiplicative forms as cocycles}, \ref{corol: passage to cover} and \ref{prop: algebroid cocycles}. Of course, we recognize the $\eta$'s from the last proposition as the bundle maps $j_{D}$ associated to Spencer operators (cf. Remark \ref{rk-convenient}); hence the relation between $\eta$ and $D$ is:
\[ \eta(\jet^1\al)=: D(\alpha) .\]
Using that 
\[\Lie_{\al}(D(\be)) = \nabla_{\jet^1\al}D(\be),\]
it is immediate that the equations on $(\eta, l)$ from the previous proposition are equivalent to the equations that ensure that $(D,l)$ is a Spencer operator.
\end{proof}

\section{Proof of Theorem \ref{t3}}
\label{sec: involutivity}



In this section we prove Theorem \ref{t3}. Let $\H \subset T\G$ be a multiplicative distribution, $(\theta, E)$ its canonically associated Pfaffian system given in 
Lemma \ref{lemma-from-H-to-theta}, and $(D,l)$ its associated Spencer operator given explicitly in Theorem \ref{t2}. Recall that 
\[ \mathfrak{g} = (\H\cap\ker \d s)|_M,\qquad E = A/\mathfrak{g}\]
and that $j_{\mathfrak{g}}$ denotes the symbol representation
\[ j_{\mathfrak{g}}: \mathfrak{g}\To \hom(TM, E), \ \ j_{\mathfrak{g}}(\beta)(X)= D_{X}(\beta).\]

As we have already pointed out in the case of contact geometry (see Subsection \ref{Contact groupoids}, especially equation (\ref{I-H})),
the involutivity of $\H$ is controlled by the bracket modulo $\H$; using $\theta$ to identify $T\G/\H$ with $t^*E$, this is
\[ I^{\H}: \H\times \H \To t^*E, \ I^{\H}(X, Y)= \theta([X, Y]).\]
We will use the following multiplicativity property of $I^{\H}$.

\begin{lemma}\label{lemma: delta theta is multiplicative}
The map $I^{\H}: \H \wedge\H \to t^{\ast}E$ satisfies
\[ I^{\H}(\d m(\xi_1, \xi_2), \d m(\xi_1',\xi_2')) =  I^{\H}(\xi_1,\xi_1')  + \Ad^\H_g I^{\H}(\xi_2,\xi_2'),\]
where $\xi_1,\xi_1'\in\H_g$, and $\xi_2,\xi_2'\in \H_h$ are such that $\d t(\xi_2) = \d s(\xi_1)$, and similarly $\d t(\xi_2') = \d s(\xi_1')$ 
\end{lemma}

\begin{proof} In general, for a regular form $u\in \Omega^1(P, F)$, denote by $I_{u}$ the resulting bilinear form
$I_u$ on $K_u:= \textrm{Ker}(u)$. If $f: Q\to P$ is a submersion, it is easy to see (using projectable vector fields) that 
$I_{f^{\ast}u}= f^{\ast}(I_u)$, i.e. 
\[ I_{f^*u}(X, Y)= I_u (\d f(X), \d f(Y))\ \ \ \ (X, Y\in K_{f^*u}= (\d f)^{-1}(K_u)).\]
In particular, $I_{m^{\ast}\theta}= m^{\ast}I_{\theta}$, $I_{\pr_{1}^{\ast}\theta}= \pr_{1}^{\ast}I_{\theta}$. A variation of this argument also gives
$I_{g^{-1}\pr_{2}^{\ast}\theta}= g^{-1}\pr_{2}^{\ast}I_{\theta}$ where $g^{-1}$ refers to the multiplication by the inverse of the first component on $E$.
Another general remark is that, for $u, v\in \Omega^1(P, E)$, $I_{u+v}= I_u+ I_v$ on $K_u\cap K_v$. Putting everything together we find that 
\[ m^{\ast}(I_{\theta}) = \pr_{1}^{\ast}(I_{\theta})+ g^{-1} \pr_{2}^{\ast}(I_{\theta})\]
on all pairs $(U, V)$ of vectors tangent to $\G_{2}$ with 
\[ U, V\in (\d\pr_1)^{-1}(\H)\cap (\d\pr_2)^{-1}(\H)\]
from which the lemma follows. 
\end{proof} \\

From the multiplicativity of $\H$, one obtains a subgroupoid of $\Jet^1\G$:
\[\Jet^1_\H\G=\{\sigma_g\in \Jet^1\G\mid \sigma_g:T_{s(g)}M\to \H_g\subset T_g\G\}.\]
Each $\sigma_g\in \Jet^{1}_{\H}\G$ induces a splitting
\[ T_{s(x)}M\oplus \B_{t(g)}\cong \H_g,\ \ (X, \alpha)\mapsto \sigma_g(X)+ R_g(\alpha).\]
and the idea is to use this splitting to analyze the vanishing of $I^{\H}$. First of all:

\begin{lemma} $I^{\H}(\H, \H^s)= 0$ if and only if $j_{\mathfrak{g}}= 0$.
\end{lemma}

\begin{proof} 
For any $y\in M$, $Y_y\in T_yM$, $\beta_y\in \mathfrak{g}_y$, extending them to sections $Y \in \Gamma(\H)$ and $\beta \in \Gamma(\H^s)$
and using them in the formula for $D$ in Theorem \ref{t2}, we have:
\[ j_{\mathfrak{g}}(\beta_y)(X_y)= D_{Y}(\beta)(y)= \theta_{1_y}([\beta, Y])= I^{H}_{y}(\beta_y, Y_y)\]
where, as before, we identify $y$ with $1_y$. 

For arbitrary $\sigma_g\in \Jet^{1}_{\H}\G$ with $s(g)= x$, $t(g)=y$ and $X_x\in T_xM$, $\beta_y\in \mathfrak{g}_y$ we write
\[ \lambda_g(X_x)= (\d t)_g(\sigma_g(X_x))= (\d m)_{(g, g^{-1})}(\sigma_g(X_x), (\d i)_g\sigma_g(X_x)),\]
\[ \beta_y= (\d m)_{(g, g^{-1})}(R_g(\beta_y)), 0_{g^{-1}})\]
and we plug them in the previous lemma:
\[ I^{\H}_{g}(\sigma_g(X_x), R_{g}(\beta_y))= I^{\H}_{y}(\lambda_g(X_x), \beta_y)= j_{\mathfrak{g}}(\beta_y)(\lambda_g(X_x)).\]
Hence $j_{\mathfrak{g}}= 0$ if and only if $I^{\H}_{g}(\sigma_g(T_xM), \H^{s}_{g})= 0$ for all $\sigma_g\in \Jet^{1}_{\H}\G$.
Note that for any $\sigma_g$ and any $\xi: T_xM\to \mathfrak{g}_y$ linear,
\[ \sigma_{g}^{\eps}(X_x)= \sigma_{g}(X_x)+ \eps R_{g}(\xi(X_x))\]
belongs to $\Jet^{1}_{\H}\G$ for $\eps$ small  enough, so the last equation also implies that $I^{\H}_{g}(\H_g, \H^{s}_{g})= 0$,
and then the equivalence with $j_{\mathfrak{g}}= 0$ is clear. 
\end{proof}\\

Next we introduce a $1$-cocycle which takes care of the value of $I^{\H}$ on expressions of type $(\sigma_g(X), \sigma_g(Y))$ and which, together with $j_{\mathfrak{g}}$, take care of the involutivity of $\H$. More precisely, define
\[ c:  \Jet^{1}_{\H}\G \To s^{\ast} \hom(\Lambda^2TM, E)\]
where, for $\sigma_g\in \Jet^{1}_{\H}\G$ with $s(g)= x$, $X_x, Y_x\in T_xM$,
\[ c(\sigma_g)
(X_x, Y_x)= \Ad^{\H}_{g^{-1}} I^{\H}_{g} (\sigma_g(X_x), \sigma_g(Y_{x})).\]

The following is now clear: 

\begin{lemma} 
$\H$ is involutive if and only if $j_{\mathfrak{g}}= 0$ and $c= 0$.
\end{lemma}

Note that, under the assumption $j_{\mathfrak{g}} \equiv 0$, $c(\sigma_g)$ only depends on $g$ and not on the entire splitting $\sigma_g$ at $g$, i.e.
$c= \pr^{*}c_0$, the pull-back along the projection $\pr: \Jet^{1}_{\H}\G\to \G$ of some
\[ c_0: \G \To s^{\ast} \hom(\Lambda^2TM, E).\]
However, even in this case, we will continue to work with $c$ because $\G$ does not act canonically on $\hom(\Lambda^2TM, E)$, and solving this problem for $c_0$ requires some work. The aim is to show now that $c$ is indeed a $1$-cocycle on $\Jet^{1}_{\H}\G$. Of course, we talk here about a $1$-cocycle along $s$ (see Remark \ref{along-s}). 
Also, the action of $\Jet^1_\H\G$ on  $\hom(\wedge^2TM, E)$ is the one induced by the representations $\lambda$, and $\Ad^{\H} \circ \pr$ on $TM$ and on $E$ respectively:
for $\sigma_{g}\in \Jet^1_\H\G$, with $s(g)= x$, $t(g)= y$ and for $T_x\in \hom(\Lambda^{2}T_{x}M, E_x)$, 
\[ g(T_x)\in \hom(\Lambda^{2}T_{y}M, E_y),\ \ g(T_x)(X_y, Y_y)= \Ad^{\H}_g T_x(\lambda_{g}^{-1}X_y, \lambda_{g}^{-1}Y_y).\]

\begin{lemma}\label{cocycle}
The map $c: \Jet^1_\H\G \to t^{\ast}\hom(\wedge^2TM, E)$ is a cocycle.
\end{lemma}

\begin{proof}
Let $\sigma_g, \sigma_h\in \Jet^{1}_{\H}\G$ composable, $X, Y\in T_{s(h)}M$. Using the formula describing the composition $\sigma_g\cdot\sigma_h$, i.e. applying (\ref{mult-i-J1}) for $u= X$ and $u= Y$ and then applying the multiplicativity of $I^{\H}$ from Lemma \ref{lemma: delta theta is multiplicative}, we find
\[\begin{aligned}
I_{gh}(\sigma_g\cdot\sigma_h(X), \sigma_g\cdot\sigma_h(Y)) & = I_g(\sigma_g(\lambda_h(X)), \sigma_g(\lambda_h(Y))) + \\ 
                                                                                                  & + \Ad^{\H}_g I_h(\sigma_h(X), \sigma_h(Y)).
\end{aligned}\]
Rewriting this in terms of $c$, we obtain, after also applying $\Ad^{\H}_{h^{-1}g^{-1}}$,
\[ c(\sigma_g\cdot\sigma_h)(X, Y)= \Ad^{\H}_{h^{-1}} c(\sigma_g)(\lambda_h(X), \lambda_h(Y))+ c(\sigma_h)(X, Y),\]
i.e. the cocycle condition $c(\sigma_g\cdot\sigma_h)= \Ad^{\H}_{h^{-1}}(c(\sigma_g))+ c(\sigma_h)$.
\end{proof}\\

Of course, the next step is to linearize $c$. Hence we pass to the Lie algebroid $\Jet^{1}_{D}A$ of $\Jet^{1}_{\H}\G$. Of course, 
the inclusion $\Jet^{1}_{\H}\G\subset \Jet^{1}\G$ induces an inclusion
\[ \Jet^{1}_{D}A\subset \Jet^{1}A .\]

\begin{lemma} Via the decomposition $\Gamma(\Jet^1(A))\cong \Gamma(A)\oplus \Omega^1(M, A)$ (cf. Example \ref{1-jets-convenient}),
$\Jet^{1}_{D}A$ consists of pairs $(\alpha, \omega)$ with the property that 
\[ D_{X}(\alpha)= -l(\omega(X)) \ \ \ \forall\ X\in \X(M).\]
Moreover, the linearization of the cocycle $c$,
\[ \kappa: \Jet^{1}_{D}A \To \hom(\Lambda^2TM, E),\]
is given on sections by
\[\kappa(\alpha,\omega)(X,Y)=-D_X(\omega (Y))+D_X(\omega (Y))+l(\omega[X,Y]).\]
\end{lemma}

\begin{proof} For the first part remark that $\Jet^{1}_{\H}\G$ is precisely the kernel of the cocycle $c$ from Proposition \ref{prop: multiplicative forms as cocycles}.
Hence its Lie algebroid is the kernel of the linearization of the cocycle, i.e. the kernel of the cocycle $\eta$ from Proposition \ref{prop: algebroid cocycles} ($k= 1$).
But, by construction of $D$, $\eta(\alpha, \omega)= D(\alpha)+ l\circ \omega$, hence the first part follows. 

For the second part, for the computations, it is better to consider
\[\tilde\theta_g = \Ad^\H_{g^{-1}}\theta_g \in \Omega^1(\G,s^{\ast}E). \]
We claim that, for any connection $\nabla$ on $E$, using the induced derivative operator $d^{\nabla}$, one has:
\begin{equation}\label{help-I} 
c(\sigma_g)(X, Y)= d^{\nabla}\tilde\theta (\sigma_g(X), \sigma_g(Y)).
\end{equation}
Indeed, Using the definition of $c$ and $\tilde\theta$, this reduces to $d^{\nabla}\tilde\theta (X, Y)= \tilde\theta ([X, Y])$ for all $X, Y\in \textrm{Ker}(\theta)$, which is clear.

We now compute the linearization $\kappa$. Let $(\alpha, \omega)$ represent a section $\zeta$ of $\Jet^{1}_{D}A$.
From the definition of $\kappa$, 
\[ \kappa(\alpha, \omega)(x) = (\d c)_{1_x}(\zeta_x).\]
Note that, thinking of elements of $\Jet^{1}_{D}A$ in terms of splittings, 
\[ \zeta_x= (\d \al)_x + \omega_x: T_xM\To T_{\alpha(x)}A, \ X\mapsto (\d \alpha)_{x}(X)+ \omega_x(X),\]
where $\omega_x(X)\in A_x$ is viewed inside $T_{\alpha(x)}A$ by the natural inclusion
\[ A_x\hookrightarrow T_{\alpha(x)}A, \ v\mapsto \frac{d}{d\eps}\big{|}_{\eps=0}(\alpha(x)+ \eps v).\]
To compute $(\d c)((\d\al)_x + \omega_x)$, we will consider the curve $\sigma^\eps_g: I \to s^{-1}(x) \subset \Jet^1\G$ given by
\[\sigma^\eps(X_x) = (\d\phi^\eps_{\al})_x(X_x + \eps\omega(X_x)),\]
for all $X_x \in T_xM$ and $\eps$ small enough (for $\phi_{\alpha}^{\epsilon}$, see Remark \ref{right-invariance}). Note that $\sigma^\eps$ is a curve in the $s$-fiber of $\Jet^1\G$ (not necessarily in $\Jet^1_\H\G$), whose derivative at $\eps = 0$ is $(\d \al)_x + \omega_x$. In fact, one has that
\[\begin{aligned}
\frac{d}{d\eps}\big{|}_{\eps=0}\sigma_g^\eps(X_x) &= \frac{d}{d\eps}\big{|}_{\eps=0}( (\d\phi^\eps_{\al})_x(X_x) + \eps\cdot (\d \phi^\eps_{\al})_x(\omega(X_x))) \\
&= (\d\al)_x(X_x) +(\d\phi^0_{\al})_x\omega(X_x)\\
& = ((\d\al)_x + \omega_x)(X_x).
\end{aligned}\]

Next, we fix a splitting of $\d s: \H \to s^{\ast}TM$, and for each $X\in\X(M)$ we denote by $\tilde{X}\in \Gamma(\H)$ the corresponding horizontal lift. Then
\[\tilde\sigma^\eps(X)_g = (\d\varphi^\eps_{\al^r})_{\varphi_{\al^r}^{-\eps}(g)}(\tilde X + \eps \omega(X)^r)\]
defines an extension of $\sigma^\eps(X_x)$ to $\X(\G)$.

From the equation (\ref{help-I}) we deduce that 
\begin{eqnarray*}\label{derivada}
\kappa(\alpha, \omega)(X_x,Y_x)= \frac{d}{d\eps}\big{|}_{\eps=0}\d^\nabla\tilde\theta_{g_\epsilon}(\tilde\sigma^\epsilon(X),\tilde\sigma^\epsilon(Y))(x),
\end{eqnarray*}
for all $X,Y\in\X(M)$.

Finally, to perform the computation, we let $\nabla$ be the pull-back via $s$ of a connection on $E$ (which we also denote by $\nabla$). We have:
\[\begin{aligned}
\d_{\nabla^s}\tilde\theta_{g_\eps}(\tilde\sigma^\eps(X),\tilde\sigma^\eps(Y))&=\nabla_{\tilde\sigma^\eps(X)}\tilde\theta(\tilde\sigma^\eps(Y))-\nabla_{\tilde\sigma^\eps(Y)}\tilde\theta(\tilde\sigma^\eps(X))-\tilde\theta([\tilde\sigma^\eps(X),\tilde\sigma^\eps(Y)])\\
&=\eps\nabla_{\tilde\sigma^\eps(X)}\tilde\theta(\d\varphi^\eps_{\alpha^r}(\omega(Y)^r))+\nabla_{\tilde\sigma^\eps(X)}\tilde\theta(\d\varphi^\eps_{\alpha^r}(\tilde Y))\\
&\quad-\eps\nabla_{\tilde\sigma^\eps(Y)}\tilde\theta(\d\varphi^\eps_{\alpha^r}(\omega(X)^r))-\nabla_{\tilde\sigma^\eps(Y)}\tilde\theta(\d\varphi^\eps_{\alpha^r}(\tilde X))\\
&\quad-\eps^2\tilde\theta([\d\varphi^\eps_{\alpha^r}(\omega(X)^r),\d\varphi^\eps_{\alpha^r}(\omega(Y)^r)])-\eps\tilde\theta(\d\varphi^\eps_{\alpha^r}[\omega(X)^r,\tilde Y])\\
&\quad+\eps\tilde\theta(\d\varphi^\eps_{\alpha^r}[\omega(Y)^r,\tilde X])-\tilde\theta(\d\varphi^\eps_{\alpha^r}[\tilde X,\tilde Y])
\end{aligned}\]
We now take the derivative when $\eps=0$ and evaluate the expression at $x$. Using the fact that $\nabla$ is the pull-back of a connection on $E$, the first term of the right hand side of the second equality gives us
\[\begin{aligned}
\nabla_{\tilde\sigma^0(X)}\tilde\theta(\omega(Y)^r)(x) = \nabla_X\tilde\theta(\omega(Y))(x)=\nabla_Xl(\omega(Y))(x),
\end{aligned}\]  
while the second term gives
\[\begin{aligned}\frac{d}{d\eps}\big{|}_{\eps=0}\nabla_{\tilde\sigma^\eps(X)}\tilde\theta(\d\varphi^\eps_{\alpha^r}(\tilde Y))(x) &=\nabla_{\frac{d}{d\eps}\tilde\sigma^\eps(X)}\tilde\theta(\tilde Y)(x)+\nabla_{\tilde X}\frac{d}{d\eps}\big{|}_{\eps=0}(\varphi^\eps_{\alpha^r})^*\tilde\theta(\tilde Y)(x)\\
&=\nabla_X\frac{d}{d\eps}\big{|}_{\eps=0}(\phi^\eps_{\alpha})^*\tilde\theta(Y)(x)\\
&=\nabla_XD_Y\alpha,
\end{aligned}\]
where in the passage from the first to the second line we have used that $\tilde\theta(\tilde Y) = 0$ because $\tilde Y\in \Gamma(\H)$. It then follows from $D_Y(\al) = -l(\omega(Y))$, that the the first line of the expression vanishes. The same argument shows also that the second line of the expression is equal to zero. So we are left with calculating the last  three terms of the expression. We obtain:
\[\begin{aligned}
-\tilde\theta([\omega(X)^r,\tilde Y])(x) + \tilde\theta([\omega(Y)^r,\tilde X])(x) - \frac{\d}{\d \eps}\big{|}_{\eps = 0}\tilde\theta(\d\varphi^\eps_{\alpha^r}[\tilde X,\tilde Y])(x).
\end{aligned}\] 
From the first two terms we obtain
\[D_Y(\omega(X))-D_X(\omega(Y)).\]
Finally, for the last term we use the fact that $\tilde X$ and $\tilde Y$ are projectable extensions of $X$ and $Y$ to obtain
\begin{eqnarray*}
\frac{d}{d\eps}\big{|}_{\eps=0}\tilde\theta(\d\varphi^\eps_{\alpha^r}[\tilde X,\tilde Y])(x)=\frac{d}{d\eps}\big{|}_{\eps=0}\tilde\theta(\d\varphi^\eps_{\alpha^r}\circ\d u([X,Y]_x))=\frac{d}{d\eps}\big{|}_{\eps=0}\tilde\theta(\d\phi^\eps_{\alpha}[X,Y]_x)\\=\frac{d}{d\eps}\big{|}_{\eps=0}(\Ad^\H_{\phi^\eps_{\alpha}})^{-1}\theta(\d\phi^\eps_{\alpha}[X,Y]_x)=D_{[X,Y]}\alpha(x)=-l(\omega([X,Y]_x)).
\end{eqnarray*}
Putting these pieces together concludes the proof of the proposition.
\end{proof}\\

We deduce:

\begin{corol} Assume that $j_{\mathfrak{g}}= 0$ and consider the induced connection $\nabla^{E}$ on $E$ ($\nabla^{E}_{X}(l(\alpha))= D_X(\alpha)$). Then 
\[ \kappa(\alpha, \omega)(X, Y)= \nabla^{E}_{X}\nabla^{E}_{Y}(\alpha)- \nabla^{E}_{Y}\nabla^{E}_{X}\alpha- \nabla^{E}_{[X, Y]}(l(\alpha)),\]
hence $k$ vanishes if and only if $\nabla^{E}$ is flat.
\end{corol}

\begin{proof} From the formula for $\kappa(\alpha, \omega)(X, Y)$ from the previous lemma, we obtain
\[ - \nabla^{E}_{X}(l\omega(Y))+ \nabla^{E}_{Y}(l\omega(X))+ l\omega([X, Y]).\]
Using that $l\circ \omega(Z)= - D_Z(\alpha)= -\nabla^{E}_{Z}(l(\alpha))$, we obtain the formula from the statement.
\end{proof}

The following closes the proof of Theorem \ref{t3}.

\begin{lemma}\label{lemma: passing to algebroid} 
If $j_{\mathfrak{g}}= 0$ and $\G$ has connected source fibers, then $c= 0$ if and only if $\kappa= 0$. 
\end{lemma}

\begin{proof}
Again, we want to apply the van Est isomorphism of Proposition \ref{prop: Van Est} and the problem is similar to the one from the previous section (Subsection \ref{sec: step 2}): the fibers of $\jet^{1}_{\H}\G$ may fail to be $1$-connected. And as there, consider the sequence of groupoids associated to $\Jet^1_\H\G$:
\[ \widetilde{\Jet^1_\H\G}\twoheadrightarrow (\Jet^1_\H\G)^0\hookrightarrow \Jet^1_\H\G \stackrel{\pr}{\To} \G,\]
and we denote by $\pr^{0}$, $\widetilde{\pr}$ the resulting maps from $(\Jet^1_\H\G)^0$ and $\widetilde{\Jet^1_\H\G}$ to $\G$. The situation is simpler here because, as
as we already remarked, $c$  as a section lives already on $\G$: $c= \pr^{\ast}(c_0)$. We can apply the van Est isomorphism to $\widetilde{\pr}^{\ast}(c_0)$, hence it suffices to 
remark that, since $\G$ is $s$-connected, $\pr^{0}$ (hence also $\widetilde{\pr}$) is surjective. This is a general fact about Lie groupoid morphisms. To see this, note that each $s$-fiber of a groupoid is principal bundle, and that a morphism of Lie groupoids induces a principal bundle map between the corresponding $s$-fibers. But the $s$-fibers of $\Jet^1_\H\G$ are mapped surjectively to the $s$-fibers of $\G$ (which are connected!), so it follows that also the restriction to the connected components of the $s$-fibers of  $\Jet^1_\H\G$ are mapped surjectively to the $s$-fibers of $\G$.
\end{proof}\\


\bibliographystyle{abbrv}
\bibliography{bibliography}

\end{document}